\documentclass[reqno]{amsart}

\usepackage{amssymb,amsfonts,amsthm,amsmath,amstext}
\usepackage{mathrsfs}  
\usepackage{mathtools}
\usepackage{dsfont}
\usepackage{cite}

\usepackage[left=2.5cm,right=2.5cm,top=1.5cm,bottom=1.5cm,includeheadfoot]{geometry}
\usepackage{hyperref,xcolor}
\hypersetup{
 pdfborder={0 0 0},
 colorlinks,
}

\usepackage{enumitem}

\allowdisplaybreaks

\newtheorem{theorem}{Theorem}
\newtheorem{remark}[theorem]{Remark}
\newtheorem{lemma}[theorem]{Lemma}
\newtheorem{proposition}[theorem]{Proposition}
\newtheorem{corollary}[theorem]{Corollary}
\newtheorem{definition}[theorem]{Definition}

\newcommand{\T}{\mathbb{T}}
\newcommand{\N}{\mathbb{N}}
\newcommand{\R}{\mathbb{R}}

\renewcommand{\l}{\left}
\renewcommand{\r}{\right}

\newcommand{\dT}{{^{\T}}{\Delta}_{a+}^{\alpha,\beta;\psi}}

\newcommand{\dTt}{{^{\T}}{\Delta}_{0+}^{\alpha,\beta;\psi}}
\newcommand{\dt}{{^{\T}}{\Delta}_{a+}^{\alpha,\beta;\psi}}
\newcommand{\dta}{{^{\T}}{\Delta}_{a+}^{\alpha,\beta;\psi(a)}}

\newcommand{\dga}{{^{\T}_{\rm RL}}{\Delta}_{a+}^{\alpha;\psi}}
\newcommand{\dgaa}{{^{\T}_{\rm RL}}{\Delta}_{a+}^{\alpha;\psi(a)}}

\newcommand{\dg}{{^{\T}_{\rm RL}}{\Delta}_{a+}^{\gamma;\psi}}
\newcommand{\dgr}{{^{\T}_{\rm RL}}{\Delta}_{a+}^{\gamma-\alpha;\psi}}

\newcommand{\dca}{{^{\T}_{\rm C}}{\Delta}_{a+}^{\alpha;\psi}}
\newcommand{\dcaa}{{^{\T}_{\rm C}}{\Delta}_{a+}^{\alpha;\psi(a)}}

\newcommand{\dm}{{^{\T}_{\rm C}}{\Delta}_{a+}^{\mu;\psi}}

\numberwithin{theorem}{section} 
\numberwithin{equation}{section}

\title[Existence, uniqueness and controllability for Hilfer differential equations]{%
Existence, uniqueness and controllability for Hilfer differential equations on times scales}

\thanks{This is a preprint of a paper whose final and definite form is published in 
'Math. Meth. Appl. Sci.', Online ISSN: 1099-1476.}

\author[J. Vanterler da C. Sousa, D. S. Oliveira, G. S. F. Frederico and D. F. M. Torres]{%
J. Vanterler da C. Sousa, D. S. Oliveira, Gast\~{a}o S. F. Frederico and Delfim F. M. Torres}

\address[J. Vanterler da C. Sousa]{Center for Mathematics, 
Computing and Cognition, Federal University of ABC,\newline 
\indent Avenida dos Estados, 5001, Bairro Bangu, 09.210-580, Santo Andre, SP, Brazil}
\email{\tt jose.vanterler@ufabc.edu.br, vanterlermatematico@hotmail.com}

\address[D. S. Oliveira]{Universidade Tecnol\'ogica Federal do Paran\'a, 
85053-525, Guarapuava-PR, Brazil}
\email{\tt oliveiradaniela@utfpr.edu.br}

\address[G. S. F. Frederico]{Federal University of Cear\'{a}, 
Campus Russas, Brazil Cear\'{a}, Brazil}
\email{\tt gastao.frederico@ua.pt}

\address[D. F. M. Torres]{Center for Research and Development in Mathematics and Applications (CIDMA),\newline
\indent Department of Mathematics, University of Aveiro, 3810-193 Aveiro, Portugal}
\email{\tt delfim@ua.pt}

\subjclass[2010]{26A33, 26E70, 34A08, 34A12, 93B05.\\
$^{*}$ Correspondent author: J. Vanterler da C. Sousa} 

\keywords{$\psi$-Hilfer fractional derivatives, 
fractional initial value problems, times scales, 
existence, uniqueness, controllability.}

\begin{document}

\begin{abstract} 
We introduce a new version of $\psi$-Hilfer fractional derivative, 
on an arbitrary time scale. The fundamental properties of the new 
operator are investigated and, in particular, we prove an integration by 
parts formula. Using the Laplace transform and the obtained integration 
by parts formula, we then propose a $\psi$-Riemann--Liouville fractional 
integral on times scales. The applicability of the new operators is 
illustrated by considering a fractional initial value problem on an 
arbitrary time scale, for which we prove existence, uniqueness 
and controllability of solutions in a suitable Banach space. The obtained 
results are interesting and nontrivial even for particular choices: 
(i) of the time scale; (ii) of the order of differentiation; 
and/or (iii) function $\psi$; opening new directions of investigation.
\end{abstract}

\maketitle

% ----------------------------------------

\section{Introduction} 

Fractional calculus is nowadays a well consolidated field of research 
with many applications in various areas of knowledge, which is interesting 
and important as it presents more refined results that are in line with reality 
\cite{Sousa,Sousa1,Sousa3,Sousa4,Sousa5,Oliveira,Podlubny,Kilbas,Lakshmikantham}. 

In recent years, applications via fractional derivatives have 
drawn the attention of numerous researchers, from different areas of knowledge, 
particularly during the two years that the world has faced the COVID-19 pandemic 
\cite{novo1}. On the other hand, mathematical models have been investigated 
for years in order to try to model and provide a better description of the studied 
phenomenon. In this sense, the fundamental role of fractional derivatives enters, 
i.e., trying to improve the results and provide results closer to reality compared 
to the classic case. In 2022 Khan et al. discussed an interesting work on a fractional 
mathematical model of tuberculosis in China and presented numerical simulations 
that show the applicability of the schemes \cite{novo2}. In the same year, 
Etemad et al. \cite{novo3}, using the Caputo-type fractional derivative, 
investigated a model of the AHIN1/09 virus. It is remarkable the importance 
and consequences that fractional derivatives provide when used to discuss 
theoretical and practical problems. For other interesting works that discuss 
phenomena, see \cite{novo4,novo6} and the references therein.

Although a well consolidated area, there are, however, still numerous open 
problems and paths to be unraveled \cite{Aghayan1,Hassani,Lopes,Aghayan}. 
A research path that some mathematicians have recently been interested consists 
to investigate fractional calculus on time scales \cite{Zhu,Yan,Kumar,Malik,Benkhettou}.

In 2007, Atici and Eloe investigated the fractional $q$-calculus on the quantum time scale \cite{Atici}. 
The authors presented a study on some properties of the fractional $q$-calculus and investigated 
the $q$-Laplace transform. In 2011, Bastos et al. investigated a class of fractional 
derivatives on times scales using the inverse Laplace transform \cite{Bastos}. Nowadays,
the subject of fractional calculus on time scales is very rich and under strong current research 
\cite{Benkhettou,Ammi,Ahmadkhanlu,Torres,Kumar1,Belaid,Williams}.
Here, we make use of the idea behind a $\psi$-Hilfer fractional derivative, 
that is, fractional differentiation of functions with respect to another function, 
which is a remarkable and relevant idea with a big impact on fractional calculus and its applications, 
in particular for problems described by fractional differential equations \cite{MR4420453}.
The study of $\psi$-Riemann--Liouville fractional integrals with respect to a function $\psi$
on time scales was initiated by Mekhalfi and Torres in 2017, where they introduced 
some generalized fractional operators on time scales 
of a function with respect to another function and carried out some applications 
to dynamic equations \cite{Mekhalfi}. Later, in 2018, Harikrishnana et al. 
proposed the $\psi$-Hilfer fractional operator on time scales as
\begin{eqnarray}
\label{901}
\dT g(x)={^{\T}}{\mathds{I}}_{a+}^{\beta(n-\alpha);\psi}
\, ^{\T}\Delta\,{^{\T}}{\mathds{I}}_{a+}^{n-\gamma;\psi}g(\tau)
\end{eqnarray}
with $\gamma=\alpha+\beta(n-\alpha)$ and 
$^{\T}\Delta=\dfrac{d}{d\tau}$ \cite{Harikrishnana}.
However, we note that (\ref{901}) does not comply with the standard
fractional calculus. In fact, instead of $^{\T}\Delta=\dfrac{d}{d\tau}$, 
we should have $^{\T}\Delta_{\psi}=\bigg(\dfrac{\Delta}{\psi^{\Delta}(x)}\bigg)$. 
To see that the term $^{\T}\Delta=\dfrac{d}{d\tau}$ is inconsistent, 
one just needs to take $\T=\mathbb{R}$ and $\beta\rightarrow 0$ or $\beta\rightarrow 1$, 
for which one does not obtain the $\psi$-Riemann--Liouville 
or the $\psi$-Caputo fractional derivative, as desired.

One of the problems when working with solutions of fractional differential equations 
on time scales is that the area is still under construction and some fundamental tools 
are still under discussion and investigation. Examples include the lack of existence, 
uniqueness, stability and controllability results, which prevent attacking 
more sophisticated problems. To fill the gap, and motivated by the above mentioned works, 
we provide here new tools and an extension of the fractional calculus on time scales. 
We claim that the results we now obtain contribute significantly to the field 
of fractional differential equations on time scales. 
Precisely, we consider the fractional initial value problem 
\begin{eqnarray}
\label{I*}
\begin{dcases}
\dTt y(t)=g^{\T}(\alpha,\gamma-\alpha)\,f(t,y(t)), \qquad t\in[0,1]=J\subseteq\T,\\
{^{\T}}{\mathds{I}}_{0}^{1-\gamma;\psi}y(0)=0, 
\end{dcases}
\end{eqnarray}
where $\dTt(\cdot)$ is the $\psi$-Hilfer fractional derivative on time scales 
of order $\alpha$, $0<\alpha<1$, and of type $0\leq{\beta}\leq{1}$ 
and $f:J\times\T\to\R$ is a right-dense continuous function.
Moreover, we will add an operator $B$ and a control function $u$
into problem (\ref{I*}), which results in
\begin{eqnarray}
\label{eqI}
\begin{dcases}
\dTt y(t)=g^{\T}(\alpha,\gamma-\alpha)\,(f(t,y(t))+Bu(t)), \quad t\in[0,1]=J\subset\T,\\
{^{\T}}{\mathds{I}}_{0}^{1-\gamma;\psi}y(0)=0,
\end{dcases}
\end{eqnarray}
where $B:\R\to\R$ is assumed to be a bounded linear operator and the control 
function $u$ belongs to $L^{2}(J,\R)$. Furthermore, 
$g^{\T}(\alpha,\gamma-\alpha)\,:=\dfrac{B^{\ T}_{0,1}(\alpha,\gamma-\alpha)}{B(\alpha,\gamma-\alpha)}$, 
where $B^{\T}_{0,1}(\alpha,\gamma-\alpha)$ and $B(\alpha,\gamma-\alpha)$ 
with $0<\alpha<1$, $\gamma=\alpha+\beta(1-\alpha)$, are time-scale Beta 
functions and the classical Beta function, respectively.

Our main contributions may be summarized in three axes.

Firstly, we present an extension to the $\psi$-Hilfer 
fractional derivative in the sense of time scales 
and discuss the essential properties in formulating 
a fractional derivative. Furthermore, some properties for the 
$\psi$-Riemann--Liouville fractional integral on time scales are also given. 
In particular, we investigate Leibniz's rule, the Laplace transform 
and integration by parts for both integral and fractional derivatives 
on time scales, respectively in the sense of $\psi$-Riemann-Liouville 
and $\psi$-Hilfer.

Our second direction of contribution consists to investigate the question 
of existence and uniqueness of solution for the initial value 
problem (\ref{I*}). In concrete, we prove the following two results.

\begin{theorem} \label{teorema32}
Suppose $J=[0,1]\subseteq\T$. 
Then, the initial value problem \eqref{I*} has a unique solution on $J$ 
if function $f(t,y(t))$ is a right-dense continuous function 
for which there exists $M>0$ such that $|f(t,y(t))|< M$ on $J$ 
and the Lipschitz condition 
\begin{equation*}
\|f(t,x)-f(t,y)\|\leq {L}\|x-y\|
\end{equation*}
holds for some ${L}>0$ and for all $t\in J$ and $x,y\in\R$.
\end{theorem}

\begin{theorem} \label{th-3.4}
Suppose $f:J\times\R\to\R$ is a right-dense continuous function 
such that there exists $M>0$ with $|f(t,y)|\leq M$ 
for all $t\in J,y\in\R.$ Then problem \eqref{I*} has a solution on $J$.
\end{theorem}

To better understand our third axle of novelty, 
let us first present our hypotheses. They are:

\noindent \rm{($A_1$)} Let $f:J\times\T\to\R$ be 
a right-dense continuous function.

\noindent \rm{($A_2$)} There exists $M>0$ 
on $[0,1]=J\subset\T$ such that 
$|f(t,y(t))|\leq M$.    

\noindent \rm{($A_3$)} Let $D=\{x\in C_{1-\gamma}(J,\R):\|x\|_{C_{1-\gamma}}\leq\rho\}$.

\noindent \rm{($A_4$)} The linear operator $\mathscr{W} u:L^{2}(J,\R)\to\R$ defined by
$$
\mathscr{W}_{\alpha}u
=\frac{1}{g^{\T}(\gamma-1,1-\gamma)\Gamma(\alpha)}\int_{0}^{t}\psi^{\Delta}(s)(\psi(t)-\psi(s))^{\alpha-1}
(Bu)(s)\Delta s
$$
has bounded invertible operators $(\mathscr{W}_{\alpha}u)^{-1}$, 
which take values in $L^{2}(J,\R)$, ${\rm Ker}$/$\mathscr{W}_{\alpha}$, 
and there exists positive constants $M_{W}$ such that 
$\|(\mathscr{W}_{\alpha})^{-1}\|\leq M_{W}$. Also, 
$B$ is a continuous operator from $\R$ to $\R$ and there exists 
a positive constant $M_{B}$ such that $\|B\|\leq M_{B}$. Furthermore, we have
$g^{\T}(\gamma-1,1-\gamma):=\dfrac{B^{\T}_{0,1}(\gamma-1,1-\gamma)}{B(\gamma-1,1-\gamma)}$ 
with $\gamma=\alpha+\beta(1-\alpha)$.

We investigate the solution controllability for \eqref{eqI}.
Precisely, we prove total controllability under hypotheses $A_i$, $i = 1,\ldots, 4$. 

\begin{theorem}
\label{th-4.2}
If all assumptions $(A_1)$--$(A_4)$ hold, then the control system \eqref{eqI} 
is controllable on $J$ provided
\begin{eqnarray}
\label{4.1}
M\,\frac{(\psi(1)-\psi(0))^{1-\beta(1-\alpha)}}{g^{\T}(\gamma-1,1-\gamma)\Gamma(\alpha+1)}<1,
\end{eqnarray}
where $g^{\T}(\gamma-1,1-\gamma):=\dfrac{B^{\T}_{0,1}(\gamma-1,1-\gamma)}{B(\gamma-1,1-\gamma)}$ 
with $\gamma=\alpha+\beta(1-\alpha)$.
\end{theorem}

It should be noted that our results are interesting and nontrivial
even in the particular cases of (i) the standard $\psi$-Hilfer fractional 
derivative (obtained by choosing $\mathbb{T}=\mathbb{R}$); 
(ii) the classical integer-order case (obtained with $\mathbb{T}=\mathbb{R}$ and $\alpha=1$); 
and (iii) the standard time-scale calculus (obtained by considering $\alpha=1$). 
Moreover, for concrete particular choices of function $\psi$, we obtain 
new formulations involving fractional derivatives on time scales. This opens 
new directions of investigation and discussion. In particular, an open problem 
consists to compute the solutions to problems involving such operators
by developing suitable numerical methods.

The article is organized as follows. In Section~\ref{sec2}, we recall 
the necessary concepts and results from the literature that help us 
throughout the paper. In Section~\ref{sec3}, we present our new extension 
involving the $\psi$-Hilfer fractional derivative on time scales. The 
fundamental properties of the new operator are investigated and, in particular, 
we obtain an integration by parts formula and the corresponding Laplace transform. 
We proceed by attacking the second main purpose of our paper, that is, 
to prove existence and uniqueness of a solution to problem \eqref{I*}. 
This is the subject of Section~\ref{sec5}. Finally, we investigate the 
controllability of (\ref{eqI}) in Section~\ref{sec6}. 
We end the article with comments and future work.

% ----------------------------------------

\section{Preliminaries}
\label{sec2}

A time scale $\T$ is an arbitrary nonempty closed subset of the real numbers. 
For $\in \T$, one defines the forward jump operator 
$\sigma: \T\rightarrow \T$ by \cite{Agarwal,Bohner,Anastassiou}
\begin{eqnarray*}
\sigma(\xi)=\inf\left\{s\in\T: s>\xi \right\}
\end{eqnarray*}
while the backward jump operator$\rho \T\rightarrow \T$ is defined by
\begin{equation*}
\rho(\xi)=\sup\left\{ s\in\T:\, s<\xi\right\}.
\end{equation*}
In addition, we put $\sigma(\max\,\T)=\max \T$ if $\max \T$ 
is finite and $\rho(\min\, \T)=\min\, \T$ if there exists a finite $\min ,\ \T$.

Following \cite{novo}, for $0\leq\gamma\leq{1}$ we define the weighted space 
$C_{1-\gamma,\psi}(J,\R)$ of continuous functions $f$ on the the finite interval 
$J=[0,1]\subset\T$ by 
\begin{equation*}
C_{1-\gamma,\psi}[0,1]
=\Big\{f:(0,1]\to\R:(\psi(t)-\psi(0))^{1-\gamma}f(t)\in C([0,1])\Big\}
\end{equation*}
with the norm
\begin{equation*}
\|f\|_{C_{1-\gamma,\psi}[0,1]}=\|(\psi(t)-\psi(0))^{1-\gamma}f(t)\|_{C_{[0,1]}}.
\end{equation*}

The derivative makes use of the set $\T^\kappa$, 
which is derived from the time scale $\T$ as follows: 
if $\T$ has left-scattered maximum $M$, then 
$\T^\kappa:=\T \setminus \{M\}$; otherwise, $\T^\kappa:=\T$.

\begin{definition}[Delta derivative \cite{Agarwal,Bohner,Anastassiou}] 
Assume $f:\T\to\mathbb{R}$ and let $t\in\T^{\kappa}$. One defines
\begin{equation*}
f^{\Delta}(t)=\lim_{s\to t}\frac{f(\sigma(s))-f(t)}{\sigma(s)-t}, 
\qquad t\neq\sigma(s),
\end{equation*}
provided the limit exists. We will call $f^{\Delta}(t)$ the delta derivative 
of $f$ at $t$. Moreover, we say that $f$ is delta differentiable on $\T^{\kappa}$ provided
$f^{\Delta}(t)$ exists for all $t\in\T^{\kappa}$. The function $f^{\Delta}:\T^{\kappa}\to\mathbb{R}$ 
is then called the (delta) derivative of $f$ on $\T^{\kappa}$.
\end{definition}

\begin{definition}[See \cite{Agarwal,Bohner,Anastassiou}] 
Let $[a,b]$ denote a closed bounded interval in $\T$. A function $F:[a,b]\to\mathbb{R}$ 
is called a delta anti-derivative of function $f:[a,b)\to\mathbb{R}$ provided $F$ 
is continuous on $[a,b]$, delta differentiable on $[a,b)$, and $F^{\Delta}(t)=f(t)$ 
for all $t\in[a,b)$. Then, we define the $\Delta$-integral of $f$ from $a$ to $b$ by 
$$
\int_{a}^{b}f(t)\Delta t:=F(b)-F(a).
$$
\end{definition}

\begin{proposition}[See \cite{Agarwal,Bohner,Anastassiou}]
\label{Proposition24} 
Suppose $a,b\in\T$, $a<b$, and $f(t)$ is continuous on $[a,b]$. Then,
$$
\int_{a}^{b}f(t)\Delta t=[\sigma(a)-a]f(a)+\int_{\sigma(a)}^{b}f(t)\Delta t.
$$
\end{proposition}

\begin{proposition}[See \cite{Agarwal,Bohner,Anastassiou}] 
\label{Prop-6} 
Suppose $\T$ is a time scale and $f$ is an increasing continuous function on $[a,b]$. 
If $F$ is the extension of $f$ to the real interval $[a,b]$ given by 
$$
F(s):=
\begin{dcases}
f(s) &{\rm if} \quad s\in\T,\\
f(t) &{\rm if} \quad s\in(t,\sigma(t))\not\subset\T,
\end{dcases}
$$
then
\begin{equation*}
\int_{a}^{b}f(t)\Delta t\leq\int_{a}^{b}F(t){\rm d}t.
\end{equation*}
\end{proposition}

Let $n-1<\alpha<n$ with $n\in\mathbb{N}$, $I=[a,b]$ be an interval such that 
$-\infty\leq{a}<b\leq\infty$ and $f,\psi\in C^{n}([a,b],\mathbb{R})$ be two 
functions such that $\psi$ is increasing and $\psi'(x)\neq{0}$ for all $x\in{I}$. 
The left-sided $\psi$-Hilfer fractional derivative 
$^{{\rm H}}\mathds{D}_{a+}^{\alpha,\beta;\psi}(\cdot)$ 
of function $f$ of order $\alpha$ and type $0\leq\beta\leq{1}$ is given by 
\begin{eqnarray}
\label{1}
^{{\rm H}}\mathds{D}_{a+}^{\alpha,\beta;\psi}f(x)
=\mathds{I}_{a+}^{\beta(n-\alpha);\psi} \bigg(\frac{1}{\psi'(x)}
\frac{{\rm d}}{{\rm d}x}\bigg)^{n}\mathds{I}_{a+}^{n-\gamma;\psi}f(x) 
\end{eqnarray}
where $\mathds{I}_{a+}^{\delta,\psi}(\cdot)$ is the Riemann--Liouville fractional 
integral with respect to another function $\psi$ with $\delta=\beta(n-\alpha)$ 
or $\delta=(1-\beta)(n-\alpha)$ and $\gamma=\alpha+\beta(n-\alpha)$.

\begin{lemma}[See \cite{Sousa}]
\label{lemma3.2}
Let $Q_1$ and $Q_2$ be two bounded sets in a Banach space $X$. Then,
\begin{enumerate}
\item $\mu(Q_1)=0$ if and only if $\overline{Q}_1$ is compact.

\item $\mu(Q_1)=\mu(\overline{Q}_1).$

\item $Q_1\subset Q_2$ implies $\mu(Q_1)\leq \mu(Q_2).$

\item $\mu(Q_1+Q_2)\leq\mu(Q_1)+\mu(Q_2)$.
\end{enumerate}
\end{lemma}

% ----------------------------------------

\section{The $\psi$-Hilfer fractional derivative on times scales}
\label{sec3}

In this section we introduce the $\psi$-Hilfer fractional derivative 
on times scales. For that we begin by recalling the notion
of $\psi$-Riemann--Liouville fractional integral 
as introduced by Mekhalfi and Torres in \cite{Mekhalfi}.

\begin{definition}[See \cite{Mekhalfi}]
\label{B} 
Suppose $\T$ is a time scale, $[a,b]$ is an interval of $\T$, $f$ is an 
integrable function on $[a,b]$, and $\psi$ is monotone having a delta derivative 
$\psi^{\Delta}$ with $\psi^{\Delta}(x)\neq{0}$ for any $x\in[a,b]$. Let $0<\alpha<1$. 
Then, the $\psi$-Riemann-Liouville fractional integral on times scales of order $\alpha$ 
of function $f$ with respect to $\psi$ is defined by
\begin{eqnarray}
\label{2}
^{{\rm \T}}\mathds{I}_{a+}^{\delta,\psi}f(x)
=\frac{1}{\Gamma(\alpha)}\int_{a}^{x}\psi^{\Delta}(s)(\psi(x)-\psi(s))^{\alpha-1}f(s)\Delta s.
\end{eqnarray}
\end{definition}

Here we make use of \eqref{2} to define a version 
for (\ref{1}) on the sense of times scales.

\begin{definition} 
Let $n-1<\alpha<n$ with $n\in\N$. Suppose $\T$ is a time scale, 
$[a,b]$ is an interval of $\T$, and $f,\psi\in C^{n}([a,b])$ are 
two functions such that $\psi$ is increasing having a delta derivative 
$\psi^{\Delta}$ with $\psi^{\Delta}(x)\neq{0}$ for all $x\in[a,b]$. 
The $\psi$-Hilfer fractional derivative on time scales of order $\alpha$ 
and type $0\leq\beta\leq{1}$ is given by
\begin{eqnarray}
\label{3}
\dT f(x)={^{\T}}{\mathds{I}}_{a+}^{\beta(n-\alpha);\psi}
\bigg(\frac{\Delta}{\psi^{\Delta}(x)}\bigg)^{(n)}{^{\T}}{\mathds{I}}_{a+}^{n-\gamma;\psi}f(x)
\end{eqnarray}
with $\gamma=\alpha+\beta(n-\alpha)$.
\end{definition}

\begin{remark}
\label{remark-1:1}
If $\T=\R$, then \eqref{3} reduces to the $\psi$-Hilfer fractional derivative \eqref{1}.
\end{remark}

\begin{remark}
\label{remark-1:2}
Taking the limit $\beta\to 0$ on both sides of \eqref{3}, we obtain the
$\psi$-Riemann--Liouville fractional derivative on time scales given by
\begin{eqnarray}
\label{4}
\dga f(x)=\,\,\bigg(\frac{\Delta}{\psi^{\Delta}(x)}\bigg)^{(n)}{^{\T}}{\mathds{I}}_{a+}^{n-\alpha;\psi}f(x). 
\end{eqnarray}
\end{remark}

\begin{remark}
\label{remark-1:3}
Taking the limit $\beta\to 1$ on both sides of \eqref{3}, 
one has the $\psi$-Caputo fractional derivative on time scales given by
\begin{eqnarray}\label{5}
\dca f(x)=\,\,{^{\T}}{\mathds{I}}_{a+}^{n-\alpha;\psi}\left(
\frac{\Delta}{\psi^{\Delta}(x)}\right)^{(n)}f(x). 
\end{eqnarray}
\end{remark}

\begin{remark}
\label{remark-1:4}
Let $\T=\R$. Taking $\alpha=1$ and $\psi(x)=x$, we get the classical derivative.
\end{remark}

\begin{remark}
\label{remark-1:5:0}
Note that \eqref{3} can be written as
\begin{eqnarray*}
\dT f(x)=\,\,{^{\T}}{\mathds{I}}_{a+}^{\gamma-\alpha;\psi}\,\dg f(x)
\end{eqnarray*}
and
\begin{eqnarray*}
\dT f(x)=\,\,\dm\,{^{\T}}{\mathds{I}}_{a+}^{n-\gamma;\psi}f(x)
\end{eqnarray*}
with $\gamma=\alpha+\beta(n-\alpha)$ and $\mu=n(1-\beta)+\beta\alpha$, 
where ${^{\T}}{\mathds{I}}_{a+}^{\gamma-\alpha;\psi}(\cdot)$
and $\dg(\cdot)$ are defined by \eqref{2} and \eqref{4}, respectively.
\end{remark}

\begin{remark}
\label{remark-1:5}
With particular choices of $\psi$ 
we obtain a wide class of fractional derivatives on time scales.
For example, let us consider the $\psi$-Hilfer fractional derivative on time scales 
and function $\psi(x)={^{\T}}{\mathds{I}}_{a+}^{(1-\beta)(n-\alpha);\psi}f(x)$. In this case, we have
\begin{equation*}
\dt f(x)={^{\T}}{\mathds{I}}_{a+}^{n-\mu;\psi}\bigg(\frac{\Delta}{\psi^{\Delta}(x)}\bigg)^{(n)}\psi(x)
\end{equation*}
with $\mu=n(1-\beta)+\beta\alpha$.
\end{remark}

\begin{definition}{\rm\cite{Benaissa}} {\rm(Beta function on time scales)} 
We define the Beta function on time scales by
$B^{\T}_{a,b}(p,q)=\displaystyle
\int_{a}^{b} (s-a)^{q-1}(b-s)^{p-1}\Delta s$, for some $p,q>0$.    
\end{definition}

\begin{proposition}{\rm\cite{Benaissa}} 
The function $B^{\T}_{a,b}(p,q)$ satisfies the inequality 
$B^{\T}_{a,b}(p,q)\geq B(p,q)(b-a)^{p+q-1}$ for some $p,q>0$, 
where $B(p,q)$ is the classical Beta function.
\end{proposition}

We proceed by proving several basic but fundamental properties. 

\begin{proposition} 
\label{A}  
For any integrable function $f$ on $[a,b]$, 
the $\psi$-Riemann--Liouville fractional integral on time scales satisfies
$$
{^{\T}}{\mathds{I}}_{a+}^{\alpha;\psi}\,{^{\T}}{\mathds{I}}_{a+}^{\beta;\psi}f(t)
\geq {^{\T}}{\mathds{I}}_{a+}^{\alpha+\beta;\psi}f(t)
$$
for any $\alpha,\beta>0$.
\end{proposition}

\begin{proof} 
Using \eqref{2} of Definition~\ref{B} yields
\begin{eqnarray*}
{^{\T}}{\mathds{I}}_{a+}^{\alpha;\psi}\,{^{\T}}{\mathds{I}}_{a+}^{\beta;\psi}f(t)
&=& \frac{1}{\Gamma(\alpha)}\int_{a}^{t}\psi^{\Delta}(s)(\psi(t)-\psi(s))^{\alpha-1}\Big( 
{^{\T}}{\mathds{I}}_{a+}^{\beta;\psi}h(s)\Big)\Delta s\\
&=&\frac{1}{\Gamma(\alpha)}\int_{a}^{t}\psi^{\Delta}(s)(\psi(t)-\psi(s))^{\alpha-1}\bigg(
\frac{1}{\Gamma(\beta)}\int_{a}^{s}\psi^{\Delta}(\tau)(\psi(s)-g(\tau))^{\beta-1}
f(\tau)\Delta\tau\bigg)\Delta s\\
&=&\frac{1}{\Gamma(\alpha)\Gamma(\beta)}\int_{a}^{t}
\int_{a}^{s}\psi^{\Delta}(s)(\psi(t)-\psi(s))^{\alpha-1}\psi^{\Delta}(\tau)
(\psi(s)-g(\tau))^{\beta-1}f(\tau)\Delta\tau\Delta s.
\end{eqnarray*}
Now, we interchange the order of integration from Fubini's theorem to obtain
\begin{eqnarray*}
{^{\T}}{\mathds{I}}_{a+}^{\alpha;\psi}\,{^{\T}}{\mathds{I}}_{a+}^{\beta;\psi}f(t)
=\frac{1}{\Gamma(\alpha)\Gamma(\beta)}\int_{a}^{t}\bigg(
\int_{\tau}^{t}\psi^{\Delta}(s)(\psi(t)-\psi(s))^{\alpha-1}
\psi^{\Delta}(\tau)(\psi(s)-g(\tau))^{\beta-1}\Delta s\bigg)f(\tau)\Delta\tau.
\end{eqnarray*}

Making the change of variable $r=\dfrac{\psi(s)-g(\tau)}{\psi(t)-g(\tau)}$, $r\in\R$,
we have ${\rm \Delta}r=\dfrac{\psi^{\Delta}(s)}{\psi(t)-\psi(\tau)}\Delta s$; 
when $s\to\tau$ one has $r\to 0$; and when $s\to t$ one has $r\to 1$. Hence,
\begin{equation*}
\begin{split}
{^{\T}}{\mathds{I}}_{a+}^{\alpha;\psi}&\,{^{\T}}{\mathds{I}}_{a+}^{\beta;\psi}f(t)\\
&=\frac{1}{\Gamma(\alpha)\Gamma(\beta)}\int_{a}^{t}\bigg[\int_{\tau}^{t}\psi^{\Delta}(s)
\bigg(1-\frac{\psi(s)-g(\tau)}{\psi(t)-g(\tau)}\bigg)^{\alpha-1}(\psi(t)-g(\tau))^{\alpha-1}
\psi^{\Delta}(\tau)(\psi(s)-g(\tau))^{\beta-1}\Delta s\bigg]f(\tau)\Delta\tau\\
&=\frac{1}{\Gamma(\alpha)\Gamma(\beta)}\int_{a}^{t}\bigg[\int_{0}^{1}(1-r)^{\alpha-1}
(\psi(t)-g(\tau))^{\alpha-1}\psi^{\Delta}(\tau)(\psi(t)-g(\tau))r^{\beta-1}
(\psi(t)-g(\tau))^{\beta-1}{\rm \Delta}r]f(\tau)\Delta\tau\\
&=\frac{1}{\Gamma(\alpha)\Gamma(\beta)}\int_{0}^{1}(1-r)^{\alpha-1}r^{\beta-1}{\rm \Delta}r
\int_{a}^{t}(\psi(t)-g(\tau))^{\alpha+\beta-1}\psi^{\Delta}(\tau)f(\tau)\Delta\tau\\
&=\frac{B^{\T}_{0,1}(\alpha,\beta)}{\Gamma(\alpha)\Gamma(\beta)}
\int_{a}^{t}(\psi(t)-g(\tau))^{\alpha+\beta-1}\psi^{\Delta}(\tau)f(\tau)\Delta\tau\\
&\geq\frac{1}{\Gamma(\alpha+\beta)}\int_{a}^{t}
(\psi(t)-g(\tau))^{\alpha+\beta-1}\psi^{\Delta}(\tau)f(\tau)\Delta\tau.
\end{split}
\end{equation*}
The proof is complete.
\end{proof}

\begin{remark}
\label{remarknovo} 
Is it possible to obtain a function $g^{\T}(\alpha,\beta)$ such that 
${^{\T}}{\mathds{I}}_{a+}^{\alpha;\psi }\,{^{\T}}{\mathds{I}}_{a+}^{\beta;\psi}f(t)
= g^{\T}(\alpha,\beta) {^{\T}}{\mathds{I}}_{a+}^{\alpha+\beta;\psi}f(t)$? 
The answer is yes. Indeed, assuming the first part of the proof of Theorem~\ref{A}, we have
\begin{eqnarray*}
{^{\T}}{\mathds{I}}_{a+}^{\alpha;\psi}\,{^{\T}}{\mathds{I}}_{a+}^{\beta;\psi}f(t)
&&=\frac{B^{\T}_{0,1}(\alpha,\beta)}{\Gamma(\alpha)\Gamma(\beta)}
\int_{a}^{t}(\psi(t)-g(\tau))^{\alpha+\beta-1}\psi^{\Delta}(\tau)f(\tau)\Delta\tau\notag\\
&&=\frac{B^{\T}_{0,1}(\alpha,\beta)}{B(\alpha,\beta)} 
\frac{1}{\Gamma(\alpha+\beta)}\int_{a}^{t}
(\psi(t)-g(\tau))^{\alpha+\beta-1}\psi^{\Delta}(\tau)f(\tau)\Delta\tau\notag\\
&&= g^{\T}(\alpha,\beta) {^{\T}}{\mathds{I}}_{a+}^{\alpha+\beta;\psi}f(t),
\end{eqnarray*}
where $B(\alpha,\beta)$ is the classical Beta function and 
$g^{\T}(\alpha,\beta):=\dfrac{B^{\T}_{0,1}(\alpha,\beta)}{B(\alpha,\beta)}$.
\end{remark}

As a consequence of Remark~\ref{remarknovo} above, taking $\T=\mathbb{R}$, we have
${\mathds{I}}_{a+}^{\alpha;\psi}\,\,{\mathds{I}}_{a+}^{\beta;\psi}f(t)
= {\mathds{I}}_{a+}^{\alpha+\beta;\psi}f(t)$.

\begin{lemma}
\label{lemma-4} 
Let $n-1\leq\gamma<n$ and $t\in C_{\gamma}[a,b]$. 
Then, $\displaystyle  {^{\T}}{\mathds{I}}_{a+}^{\alpha;\psi}f(a)
=\lim_{x\to a+}{^{\T}}{\mathds{I}}_{a+}^{\alpha;\psi}f(x)=0$, $n-1\leq\gamma<\alpha$.
\end{lemma}

\begin{proof} 
First, we have that 
${^{\T}}{\mathds{I}}_{a+}^{\alpha;\psi}f(x)\in C_{\gamma}[a,b]$ is bounded. 
Since $f\in C_{\gamma}[a,b]$, then $(\psi(x)-\psi(a))^{\gamma}f(x)$ is continuous 
on $[a,b]$ and thus
\begin{eqnarray}
\label{*}
\Big|(\psi(x)-\psi(a))^{\gamma}f(x)\Big|< M 
\quad \Rightarrow 
\quad |f(x)|<\Big|(\psi(x)-\psi(a))^{-\gamma}\Big| M,
\end{eqnarray}
$x\in[a,b]$, for some positive constant $M$. Taking 
${^{\T}}{\mathds{I}}_{a+}^{\alpha;\psi}(\cdot)$ on both sides of \eqref{*} yields
\begin{eqnarray*}
\Big|{^{\T}}{\mathds{I}}_{a+}^{\alpha;\psi}f(x)\Big|
&<&\Big|{^{\T}}{\mathds{I}}_{a+}^{\alpha;\psi}
(\psi(x)-\psi(a))^{-\gamma}\Big|M\\
&=&M\frac{\Gamma(n-\gamma)}{\Gamma(\alpha+n-\gamma)}
(\psi(x)-\psi(a))^{\alpha-\gamma}.
\end{eqnarray*}
Since $\gamma<\alpha$, the right-sided tends to $0$ as $x$ tends to $a+$. Therefore, one has
$$
{^{\T}}{\mathds{I}}_{a+}^{\alpha;\psi}f(a)=\lim_{x\to a+}{^{\T}}{\mathds{I}}_{a+}^{\alpha;\psi}f(x)=0
$$
and the result is proved.
\end{proof}

\begin{theorem} 
Let $0<\alpha<1$ and $\psi$ be monotone having a delta derivative 
$\psi^{\Delta}$ with $\psi^{\Delta}(x)\neq{0}$ for all $x\in[a,b]$. 
The $\psi$-Riemann--Liouville fractional integral on time scales 
is a bounded operator given by
\begin{equation*}
\Big\| {^{\T}}{\mathds{I}}_{a+}^{\alpha;\psi} 
f\Big\|_{C_\gamma,\psi}\leq {L}\big\|f\big\|_{C_\gamma,\psi}
\end{equation*}
with $\displaystyle {L}=\frac{(\psi(b)-\psi(a))^{\alpha}}{\Gamma(\alpha+1)}$.
\end{theorem}

\begin{proof} 
From \eqref{2} of Definition~\ref{B} and using Proposition~\ref{Prop-6}, it follows that
\begin{eqnarray*}
\Big\|{^{\T}}{\mathds{I}}_{a+}^{\alpha;\psi} f\Big\|_{C_\gamma,\psi}
&=&\max_{x\in[a,b]}\Big|(\psi(x)-\psi(a))^{\gamma}\,{^{\T}}{\mathds{I}}_{a+}^{\alpha;\psi} f(x)\Big|\\
&=&\max_{x\in[a,b]}\bigg|(\psi(x)-\psi(a))^{\gamma}\frac{1}{\Gamma(\alpha)}\int_{a}^{x}
\psi^{\Delta}(s)(\psi(t)-\psi(s))^{\alpha-1}f(s)\Delta s\bigg|\\
&\leq &\|f\|_{C_{\gamma,\psi}}\bigg|\frac{1}{\Gamma(\alpha)}\int_{a}^{x}
\psi^{\Delta}(s)(\psi(t)-\psi(s))^{\alpha-1}\Delta s\bigg|\\
&\leq &\frac{\|f\|_{C_{\gamma,\psi}}}{\Gamma(\alpha)}\frac{(\psi(b)-\psi(a))^{\alpha}}{\alpha}\\
&=&\frac{(\psi(b)-\psi(a))^{\alpha}}{\Gamma(\alpha+1)}\|f\|_{C_{\gamma,\psi}}\\
&=&{L}\|f\|_{C_{\gamma,\psi}},
\end{eqnarray*}
where $\displaystyle {L}=\frac{(\psi(b)-\psi(a))^{\alpha}}{\Gamma(\alpha+1)}$.
\end{proof}

\begin{lemma}
\label{lemma1} 
Let $\T$ be a time scale, $(a,b]$ with $-\infty\leq{a}<b\leq\infty$ 
be an interval in the real line, $\alpha>0$, and $\psi(x)$ be a monotone 
increasing and positive function in $\Delta$ sense whose derivative 
is continuous in $(a,b]$. Then,
\begin{eqnarray}
\label{8}
{^{\T}}{\mathds{I}}_{a+}^{\alpha;\psi}f(x)
=\sum_{n=0}^{\infty}\binom{-\alpha}{n}{f^{\Delta}}^{(n)}(x) 
\frac{(\psi(x)-\psi(a))^{\alpha+n}}{\Gamma(\alpha+n+1)} 
\end{eqnarray}
where ${f^{\Delta}}^{(n)}(\cdot)$ is the $n$th
derivative on the time scale $\T$ and $x>a$.
\end{lemma}

\begin{proof} 
Let us consider $f$ as follows:
\begin{eqnarray}
\label{9}
f(t)=\sum_{n=0}^{\infty}\frac{{f^{\Delta}}^{(n)}(x)}{n!}(\psi(t)-\psi(x))^{n}.
\end{eqnarray}
Taking ${^{\T}}{\mathds{I}}_{a+}^{\alpha;\psi}$ on both sides of (\ref{9}) yields
\begin{eqnarray*}
{^{\T}}{\mathds{I}}_{a+}^{\alpha;\psi}f(x)
&=&\frac{1}{\Gamma(\alpha)}\int_{a}^{x} \psi^{\Delta}(t)(\psi(x)-\psi(t))^{\alpha-1}f(t)\Delta t\\
&=&\frac{1}{\Gamma(\alpha)}\int_{a}^{x}\psi^{\Delta}(t)(\psi(x)-\psi(t))^{\alpha-1}
\Bigg(\sum_{n=0}^{\infty}\frac{{f^{\Delta}}^{(n)}(x)}{n!}(\psi(t)-\psi(x))^{n}\Bigg)\Delta t\\
&=&\sum_{n=0}^{\infty}\frac{{f^{\Delta}}^{(n)}(x)}{n!}
\frac{(-1)^n}{\Gamma(\alpha)}\int_{a}^{x}\psi^{\Delta}(t)(\psi(x)-\psi(t))^{\alpha+n-1}\Delta t\\
&=&\sum_{n=0}^{\infty}\frac{{f^{\Delta}}^{(n)}(x)}{n!}\frac{(-1)^n}{\Gamma(\alpha)}
\frac{(\psi(x)-\psi(a))^{\alpha+n}}{\Gamma(\alpha+n+1)}\Gamma(\alpha+n).
\end{eqnarray*}
Taking into account the identity
\begin{equation*}
\binom{\alpha}{n}=\frac{(-1)^{n}\alpha\Gamma(n-\alpha)}{\Gamma(1-\alpha)\Gamma(n+1)},
\end{equation*}
we have
\begin{equation*}
\binom{-\alpha}{n}=\frac{(-1)^{n-1}(-\alpha)\Gamma(\alpha+n)}{\Gamma(1+\alpha)\Gamma(n+1)}
= \frac{(-1)^{n}\Gamma(\alpha+n)}{\Gamma(\alpha)\Gamma(n+1)}.
\end{equation*}
We conclude that
\begin{equation*}
{^{\T}}{\mathds{I}}_{a+}^{\alpha;\psi}f(x)
=\sum_{n=0}^{\infty}\binom{-\alpha}{n}{f^{\Delta}}^{(n)}(x) 
\frac{(\psi(x)-\psi(a))^{\alpha+n}}{\Gamma(\alpha+n+1)}
\end{equation*}
and the proof is complete.
\end{proof}

Now, we present the Leibniz rule associated with the $\psi$-Riemann--Liouville 
fractional integral on time scales.

\begin{theorem}
\label{A*} 
Let $\T$ be a time scale, $(a,b]$ with $-\infty\leq{a}<b\leq\infty$ be an interval in the real line, 
$\alpha>0$, and $\psi(x)$ be a monotone increasing and positive function in $\Delta$ sense
whose derivative is continuous in $(a,b]$. The left fractional integral of the
product of two functions is given by
\begin{equation*}
{^{\T}}{\mathds{I}}_{a+}^{\alpha;\psi}(fh)(x)
=\sum_{k=0}^{\infty}{f^{\Delta}}^{(k)}(x) {^{\T}}{\mathds{I}}_{a+}^{\alpha;\psi} h(x),
\end{equation*}
where ${f^{\Delta}}^{(k)}$ is the $k$th derivative on the time scale $\T$ and $x>a$.
\end{theorem}

\begin{proof} 
Let $f$ and $h$ satisfying the condition of Lemma~\ref{lemma1}. 
Then, from \eqref{8} it follows that
\begin{equation*}
{^{\T}}{\mathds{I}}_{a+}^{\alpha;\psi}(fh)(x)
=\sum_{m=0}^{\infty}\binom{-\alpha}{m} {(fh)^{\Delta}}^{(m)}(x)
\frac{(\psi(x)-\psi(a))^{\alpha+m}}{\Gamma(\alpha+m+1)}.
\end{equation*}

Using the Leibniz rule, 
\begin{equation*}
{(fh)^{\Delta}}^{(m)}(x)=\sum_{k=0}^{m}{f^{\Delta}}^{(k)}(x){h^{\Delta}}^{(m-k)}(x)
\end{equation*}
with $m\in\N$ and $f,h\in C^{m}([a,b])$, which yields that
\begin{eqnarray*}
{^{\T}}{\mathds{I}}_{a+}^{\alpha;\psi}(fh)(x)
&=&\sum_{m=0}^{\infty}\binom{-\alpha}{m}
\sum_{k=0}^{m}\binom{m}{k}{f^{\Delta}}^{(k)}(x){h^{\Delta}}^{(m-k)}(x)
\frac{(\psi(x)-\psi(a))^{\alpha+m}}{\Gamma(\alpha+m+1)}\\
&=&\sum_{k=0}^{\infty}{f^{\Delta}}^{(k)}(x)\sum_{m=k}^{\infty}
\binom{-\alpha}{m}\binom{m}{k}{h^{\Delta}}^{(m-k)}(x)
\frac{(\psi(x)-\psi(a))^{\alpha+m}}{\Gamma(\alpha+m+1)}.
\end{eqnarray*}
Considering $n=m-k$ and using the identity 
\begin{equation*}
\binom{-\alpha}{n+k}\binom{n+k}{k}=\binom{-\alpha}{k}\binom{-(\alpha+k)}{n}
\end{equation*}
we obtain that
\begin{eqnarray*}
{^{\T}}{\mathds{I}}_{a+}^{\alpha;\psi}(fh)(x)
&=&\sum_{k=0}^{\infty}{f^{\Delta}}^{(k)}(x)
\sum_{n=0}^{\infty}\binom{-\alpha}{n+k}
\binom{n+k}{k}{h^{\Delta}}^{(n+k-k)}(x)
\frac{(\psi(x)-\psi(a))^{\alpha+n+k}}{\Gamma(\alpha+n+k+1)}\\
&=&\sum_{k=0}^{\infty}{f^{\Delta}}^{(k)}(x)\binom{-\alpha}{k}\sum_{n=0}^{\infty}
\binom{-(\alpha+k)}{n}{h^{\Delta}}^{(n)}(x)\frac{(\psi(x)-\psi(a))^{\alpha+n+k}}{\Gamma(\alpha+n+k+1)}\\
&=&\sum_{k=0}^{\infty}\binom{-\alpha}{k}{f^{\Delta}}^{(k)}(x)\,
{^{\T}}{\mathds{I}}_{a+}^{\alpha+k;\psi}h(x)
\end{eqnarray*}
and the proof is complete.
\end{proof}

\begin{proposition}  
Let $0\leq \alpha \leq 1$ and $0\leq \beta \leq 1$. Then,
\begin{enumerate}[label={\normalfont(\arabic*)}]
\item $\dt\big(\lambda_1 f(t)+\lambda_2 \psi(t)\big)=\lambda_1\dt f(t)
+\lambda_2\dt \psi(t)$, where $\lambda_{1},\lambda_{2}\in\mathbb{R}$.

\item $\displaystyle \dt(\psi(x)-\psi(a))^{\delta-1}
=\frac{\Gamma(\delta)}{\Gamma(\delta-\alpha)} 
(\psi(x)-\psi(a))^{\delta-\alpha-1}$, $\delta>1$.
\end{enumerate}
\end{proposition}

\begin{proof}
(1) Using the fact that 
$\dt f(t)={^{\T}}{\mathds{I}}_{a+}^{\gamma-\alpha;\psi}\,\dg f(t)$ and 
because ${^{\T}}{\mathds{I}}_{a+}^{\gamma-\alpha;\psi}(\cdot)$ and $\dg(\cdot)$ 
are linear, we have that $\dt(\cdot)$ is also linear. 

\noindent (2) Remembering that $\dg(\psi(x)-\psi(a))^{\delta-1}
=\dfrac{\Gamma(\delta)}{\Gamma(\delta-\alpha)} (\psi(x)-\psi(a))^{\delta-\alpha-1}$ 
and 
$$
{^{\T}}{\mathds{I}}_{a+}^{\alpha;\psi}(\psi(x)-\psi(a))^{\delta-1}
=\dfrac{\Gamma(\delta)}{\Gamma(\delta+\alpha)} (\psi(x)-\psi(a))^{\delta+\alpha-1},
$$ 
we obtain that
\begin{eqnarray*}
\dt(\psi(x)-\psi(a))^{\delta-1}
&=&{^{\T}}{\mathds{I}}_{a+}^{\gamma-\alpha;\psi}\,\dg(\psi(x)-\psi(a))^{\delta-1}\\
&=&{^{\T}}{\mathds{I}}_{a+}^{\gamma-\alpha;\psi}\bigg\{
\frac{\Gamma(\delta)}{\Gamma(\delta-\gamma)}
(\psi(x)-\psi(a))^{\delta-\gamma-1}\bigg\}\\
&=&\frac{\Gamma(\delta)}{\Gamma(\delta-\alpha)}(\psi(x)-\psi(a))^{\delta-\alpha+1}.
\end{eqnarray*}
The proof is complete.
\end{proof}

\begin{remark}
\label{rema1} 
In particular, given $1\leq k\in \mathbb{N}$, and as $\delta>1$, 
we have $\displaystyle \dt(\psi(x)-\psi(a))^{k}
=\frac{k!}{\Gamma(k-1-\alpha)} (\psi(x)-\psi(a))^{k-\alpha}$. 
On the other hand, for $1>k\in \mathbb{N}_{0}$, 
we have $\displaystyle \dt(\psi(x)-\psi(a))^{k}=0$.
\end{remark}

\begin{theorem} 
\label{th-5}
If $f\in C^{n}[a,b]$, $n-1<\alpha<n$, and $0\leq\beta\leq 1$, then
\begin{equation*}
{^{\T}}{\mathds{I}}_{a+}^{\gamma;\psi}\,\dt f(x)
= g^{\T}(\alpha,\gamma-\alpha) g^{\T}(\gamma-n,n-\gamma) f(x)- g^{\T}(\alpha,\gamma-\alpha)\sum_{k=1}^{n}\frac{(\psi(x)-\psi(a))^{\gamma-k}}{\Gamma(\gamma-k+1)}{
f_{\psi}^{\Delta}}^{(n-k)}\,{^{\T}}{\mathds{I}}_{a+}^{n-\gamma;\psi}f(a),
\end{equation*}
\end{theorem}

\begin{proof} 
Using the identity
\begin{eqnarray}
\label{star}
{^{\T}}{\mathds{I}}_{a+}^{\alpha;\psi}\,{^{\T}}{\mathds{I}}_{a+}^{\gamma;\psi}f(x)
=g^{\T}(\alpha,\beta) {^{\T}}{\mathds{I}}_{a+}^{\alpha+\gamma;\psi}f(x),
\end{eqnarray}
we have 
\begin{eqnarray}
\label{star-2}
{^{\T}}{\mathds{I}}_{a+}^{\alpha;\psi}\,\dt f(x)
=g^{\T}(\alpha,\gamma-\alpha){^{\T}}{\mathds{I}}_{a+}^{\gamma;\psi}\, \dg f(x) 
\end{eqnarray}
with $\gamma=\alpha+\beta(n-\alpha)$. Integrating by parts $n$-times, we get
\begin{eqnarray*}
{^{\T}}{\mathds{I}}_{a+}^{\gamma;\psi}\,\dg f(x)
=g^{\T} (\gamma-n,n-\gamma) f(x)-\sum_{k=1}^{n}
\frac{(\psi(x)-\psi(a))^{\gamma-k}}{\Gamma(\gamma-k+1)}{f_{\psi}^{\Delta}}^{(n-k)}
\,{^{\T}}{\mathds{I}}_{a+}^{n-\gamma;\psi}f(a).
\end{eqnarray*}

From \eqref{star} and \eqref{star-2}, we conclude that
\begin{equation*}
{^{\T}}{\mathds{I}}_{a+}^{\gamma;\psi}\,\dt f(x)
= g^{\T}(\alpha,\gamma-\alpha) g^{\T}(\gamma-n,n-\gamma) f(x)- g^{\T}(\alpha,\gamma-\alpha)\sum_{k=1}^{n}\frac{(\psi(x)-\psi(a))^{\gamma-k}}{\Gamma(\gamma-k+1)}{
f_{\psi}^{\Delta}}^{(n-k)}\,{^{\T}}{\mathds{I}}_{a+}^{n-\gamma;\psi}f(a),
\end{equation*}
where
\begin{equation*}
{f_{\psi}^{\Delta}}^{(n)}
:=\l(\frac{\Delta}{\psi^{\Delta}(x)}\r)^{n}f(x)
\end{equation*}
and $g^{\T}(p,q)=\dfrac{B_{0,1}^{\T}(p,q)}{B(p,q)}$ 
are the Beta functions in time scales and the classical Beta function. 
The result is proved.
\end{proof}

\begin{remark} 
\begin{enumerate}
\item  Taking $n=1$ in Theorem~\ref{th-5}, we have
\begin{equation*}
{^{\T}}{\mathds{I}}_{a+}^{\gamma;\psi}\,\dt f(x)
= g^{\T}(\alpha,\gamma-\alpha) g^{\T}(\gamma-1,1-\gamma) f(x)
- g^{\T}(\alpha,\gamma-\alpha)\frac{(\psi(x)-\psi(a))^{\gamma-1}}{\Gamma(\gamma)} 
\,{^{\T}}{\mathds{I}}_{a+}^{1-\gamma;\psi}f(a).
\end{equation*}

\item Taking $\T=\mathbb{R}$ in Theorem~\ref{th-5}, we have
\begin{equation*}
{\mathds{I}}_{a+}^{\gamma;\psi}\, \Delta^{\alpha,\beta;\psi}_{a+}
f(x) =  f(x)- \frac{(\psi(x)-\psi(a))^{\gamma-1}}{\Gamma(\gamma)} 
\,{\mathds{I}}_{a+}^{1-\gamma;\psi}f(a).
\end{equation*}
\end{enumerate}
\end{remark}

\begin{proposition}
\label{prop-I} 
For any integrable function $h$ on $[a,b]$ one has
\begin{equation*}
\dt\,{^{\T}}{\mathds{I}}_{a+}^{\alpha;\psi}f(x)=g^{\T}(\gamma-n,n-\gamma)\,\, f(x).
\end{equation*}
\end{proposition}

\begin{proof} 
By definition of $\dt(\cdot)$, and using Proposition~\ref{A}, 
Lemma~\ref{lemma-4} and Theorem~\ref{th-5}, we can write that
\begin{eqnarray*}
\dt\,{^{\T}}{\mathds{I}}_{a+}^{\alpha;\psi}f(x)
&=&{^{\T}}{\mathds{I}}_{a+}^{\gamma-\alpha;\psi}\,
\dgr f(x)\\
&=&g^{\T}(\gamma-n,n-\gamma) f(x)-\sum_{k=1}^{n}\frac{(\psi(x)-\psi(a))^{\gamma-k}}{\Gamma(\gamma
-k+1)}{f_{\psi}^{\Delta}}^{(n-k)}\,{^{\T}}{\mathds{I}}_{a+}^{n-\gamma;\psi}f(a)\\
&=&g^{\T}(\gamma-n,n-\gamma) f(x).
\end{eqnarray*}
The proof is complete.
\end{proof}

\begin{theorem}
\label{th-6} 
The $\psi$-Hilfer fractional derivative on time scales is a bounded operator 
for all $n-1<\alpha<n$ and $0\leq\beta\leq{1}$ with
\begin{equation*}
\Big\|\dt f\Big\|_{C_{\gamma,\psi}}
\leq {L} \Big\|{f^{\Delta}}^{(n)}\Big\|_{C_{\gamma,\psi}^{n}}.
\end{equation*}
\end{theorem}

\begin{proof} 
Remembering that $\dt f(x)={^{\T}}{\mathds{I}}_{a+}^{\gamma-\alpha;\psi}\,\dg f(x)$ yields
\begin{eqnarray*}
\Big\|\dt f\Big\|_{C_{\gamma,\psi}}
&=&\Big\|{^{\T}}{\mathds{I}}_{a+}^{\gamma-\alpha;\psi}\,
\dg f(x)\Big\|\\
&\leq &\frac{\Big\|\dg f\Big\|_{C_{\gamma,\psi}}}{\Gamma(\gamma-\alpha)}\max_{x\in[a,b]}
\bigg|\int_{a}^{x}\psi^{\Delta}(t)(\psi(x)-\psi(t))^{\gamma-\alpha-1}\Delta t\bigg|\\
&\leq &\frac{(\psi(b)-\psi(a))^{\gamma-\alpha}}{(\gamma-\alpha)\Gamma(\gamma-\alpha)}
\Big\|\dg f\Big\|_{C_{\gamma,\psi}}\\
&\leq &\frac{(\psi(b)-\psi(a))^{\gamma-\alpha}}{(\gamma-\alpha)\Gamma(\gamma-\alpha)}
\frac{\Big\|{f^{\Delta}}^{(n)}\Big\|_{C_{\gamma,\psi}^{n}}}{\Gamma(n-\gamma)}
\max_{x\in[a,b]}\bigg|\int_{a}^{x}\psi^{\Delta}(t)(\psi(x)-\psi(t))^{n-\gamma-1}\Delta t\bigg|\\
&\leq &\frac{(\psi(b)-\psi(a))^{n-\alpha}}{(n-\gamma)(\gamma-\alpha)
\Gamma(n-\gamma)\Gamma(\gamma-\alpha)}\Big\|{f^{\Delta}}^{(n)}\Big\|_{C_{\gamma,\psi}^{n}},
\end{eqnarray*}
which proves the intended result.
\end{proof}

\begin{theorem}
\label{th-11} 
Let $f\in C^{1}([a,b])$, $\alpha\geq{0}$, $\delta\geq{0}$ 
and $0\leq{\beta}\leq{1}.$ Then, 
\begin{equation*}
\dt\,{^{\T}}{\mathds{I}}_{a+}^{\delta;\psi}f(x)
=g^{\T}(1-\alpha,\delta)\,\,g^{\T}(\gamma-\alpha,
\gamma-\delta){\mathds{I}}_{a+}^{2\gamma-\alpha-\delta;\psi}f(x)
\end{equation*}
with $\alpha\geq\delta\geq{0}$.
\end{theorem}

\begin{proof} 
We begin by noting that
\begin{equation}
\label{I}
\begin{split}
\dga\,{^{\T}}{\mathds{I}}_{a+}^{\delta;\psi}f(x)
&=\l(\frac{\Delta}{\psi^{\Delta}(x)}\r)
{^{\T}}{\mathds{I}}_{a+}^{1-\alpha;\psi}\,{^{\T}}{\mathds{I}}_{a+}^{\delta;\psi}f(x)\\
&=g^{\T}(1-\alpha,\delta)\,{^{\T}}{\mathds{I}}_{a+}^{\alpha-\delta;\psi}f(x).
\end{split}
\end{equation}
Using the relation
\begin{equation*}
\dt f(x)={^{\T}}{\mathds{I}}_{a+}^{\alpha-\delta;\psi}\,\dga f(x)
\end{equation*}
with $\gamma=\alpha+\beta(1-\alpha)$ and (\ref{I}), we get that
\begin{eqnarray*}
\dt\,{^{\T}}{\mathds{I}}_{a+}^{\delta;\psi}f(x)
&=&{^{\T}}{\mathds{I}}_{a+}^{\gamma-\alpha;\psi}\,
\dg\,{^{\T}}{\mathds{I}}_{a+}^{\delta;\psi}f(x)\\
&=&g^{\T}(1-\alpha,\delta)\,\,g^{\T}(\gamma-\alpha,\gamma
-\delta){\mathds{I}}_{a+}^{2\gamma-\alpha-\delta;\psi}f(x),
\end{eqnarray*}
which proves the intended equality.
\end{proof}

We now approach the $\Delta$-Laplace 
transform on times scales and integration by parts.

Let $p\in R(\T,\mathbb{R})$. We define the exponential function by
\begin{equation*}
e_{p}(t,s)=exp\left( \int_{a}^{t} x_{\mu(\tau)} (p(\tau))\Delta\tau \right).
\end{equation*}

\begin{definition} 
Let $f,\psi:[0,\infty)\rightarrow \mathbb{R}$ be real valued functions 
such that $\psi$ is a nonnegative increasing function with $\psi(0)=0$. 
Then the Laplace transform of $f$ with respect to $\psi$ is defined by
\begin{equation*}
\mathscr{L}_{\psi}(f(t))=F(s)=\int_{0}^{\infty} e^{-s \psi(t)} \psi'(t) f(t) dt
\end{equation*}
for all $s\in \mathbb{C}$ for which this integral converges. 
Here, $\mathscr{L}_{\psi}(\cdot)$ denotes the Laplace transform with respect 
to $\psi$, which we call the generalized Laplace transform.
\end{definition}

\begin{corollary}
\label{cor-20}
If $f(t)$ is a function whose classical Laplace transform is $F(s)$, 
then the generalized Laplace transform of function 
$f\circ \psi=f(\psi(t))$ is also $F(s)$:
\begin{equation*}
\mathscr{L}[f(t)]=F(s) \quad \Rightarrow \quad \mathscr{L}_{\psi}[f(\psi(t))]=F(s).
\end{equation*}
\end{corollary}

\begin{definition}
For $f:\T\to\R$, the time-scale or generalized transform of $f$, 
denoted by $\mathscr{L}[f]$ or $F(z)$, is given by
\begin{equation*}
^{\rm{\T}}\mathscr{L}[f](z)=F(z):=\int_{0}^{\infty}f(t)\psi^{\sigma}(t)\Delta t,
\end{equation*}
where $\psi(t)={\rm e}_{\theta z}(t,0)$
$\big(\psi^{\sigma}(t)={\rm e}_{\theta z}(\sigma(t))\big)$, that is,
\begin{equation*}
^{\rm{\T}}\mathscr{L}[f](z)=F(z)
:=\int_{0}^{\infty}f(t){\rm e}_{\theta z}(\sigma(t))\Delta t.
\end{equation*}
\end{definition}

\begin{theorem}[Inversion of the transform]
Suppose that $F(z)$ is analytic in the region 
${\rm Re}_{\mu}(z)>{\rm Re}_{\mu}(c)$ and $F(z)\to 0$ 
as $|z|\to\infty$ in this region. Furthermore, suppose $F(z)$ 
has finitely many regressive poles of finite order $\{z_1,z_2,\ldots,z_n\}$ 
and $\tilde{F}_{\R}(z)$ is the transform of the function $\tilde{f}(t)$ on $\R$ 
that correspondents to the transform $F(z)=F_{\T}(z)$ of $f(t)$ on $\T$. If
\begin{equation*}
\int_{c-i\infty}^{c+i\infty}|\tilde{F}_{\R}(z)||{\rm d}z|<\infty,
\end{equation*}
then
\begin{equation*}
f(t)=\sum_{i=1}^{n}{\rm Res}_{z=z_i}{\rm e}_{z}(t,0)F(z)
\end{equation*}
has transform $F(z)$ for all $z$ with ${\rm Re}(z)>c$.
\end{theorem}

Our main purpose here is to propose an extension 
of the Laplace transform on time scales using $\psi$.

The operators $\dt$, $\dga$, $\dca$ and ${^{\T}}{\mathds{I}}_{a+}^{\alpha;\psi}$ 
can be written as the conjugation of the standard fractional operators
with the operation of composition with $\psi$ or $\psi^{-1}$, given by
\begin{equation}
\label{90}
\begin{split}
{^{\T}}{\mathds{I}}_{a+}^{\alpha;\psi}
&=Q_{\psi}\circ{^{\T}}{\mathds{I}}_{a+}^{\alpha;\psi(a)}
\circ(Q_{\psi})^{-1},\\
\dga &=Q_{\psi}\circ \dgaa\circ(Q_{\psi})^{-1},\\
\dca &=Q_{\psi}\circ \dcaa\circ(Q_{\psi})^{-1},
\end{split}
\end{equation}
and
$$
\dt=Q_{\psi}\circ \dta\circ(Q_{\psi})^{-1},
$$
where the functional operator $Q_{\psi}$ is given by
\begin{equation*}
(Q_{\psi} f)(x)=f(\psi(x)).
\end{equation*}

\begin{definition} 
Let $f,\psi:\T\to\R$ be such that $\psi$ is a nonnegative increasing function with $\psi(0)=0$. 
Then, the time scale generalized transform of $f$ with respect to $\psi$ is defined by
\begin{equation*}
^{\rm{\T}}\mathscr{L}[f](z)=F(z)
:=\int_{0}^{\infty}f(t)\psi^{\Delta}(t)\psi^{\sigma}(t)\Delta t,
\end{equation*}
where $\psi(\psi(t))={\rm e}_{\theta z}(\psi(t),0)$
$(\psi^{\sigma}(\psi(t))={\rm e}_{\theta z}(\sigma(\psi(t)))$, that is,
\begin{equation*}
^{\rm{\T}}\mathscr{L}[f](z)=F(z)
:=\int_{0}^{\infty}f(t)\psi^{\Delta}(t){\rm e}_{\theta z}(\sigma(\psi(t)))\Delta t.    
\end{equation*}
\end{definition}

Next, we prove an integration by parts formula
for the $\psi$-Riemann-Liouville fractional integral on times scales.

\begin{theorem}
\label{th-4}
Let $\alpha>0$, $p,q\geq{1}$ and $\frac{1}{p}+\frac{1}{q}\leq 1+\alpha$, 
where $p\neq{1}$ and $q=1+n$ in the case when $\frac{1}{p}+\frac{1}{q}=1+\alpha$. 
Moreover, let
\begin{equation*}
{^{\T}}{\mathds{I}}_{a+}^{\alpha;\psi}(L_p)
=\Big\{f:f={^{\T}}{\mathds{I}}_{a+}^{\alpha;\psi}g,\,\, 
g\in L_p(a,b)\Big\}.
\end{equation*}
The following integration by parts formulas hold: 
if $\varphi\in L_p(a,b)$ and $\phi\in L_q(a,b)$, then
\begin{equation*}
\int_{a}^{b}\bigg({^{\T}}{\mathds{I}}_{a+}^{\alpha;\psi}\phi(t)\bigg)\varphi(t)\Delta t
= \int_{a}^{b}\phi(t)\psi^{\Delta}(t)\,{^{\T}}{\mathds{I}}_{b-}^{\alpha;\psi}
\l(\frac{\varphi(t)}{\psi^{\Delta}(t)}\r) \Delta t.
\end{equation*}
\end{theorem}

\begin{proof} 
If $\varphi\in L_{p}(a,b)$ and $\phi\in L_q(a,b)$, 
then, from \eqref{2} of Definition~\ref{B}, it follows that
\begin{eqnarray*}
\int_{a}^{b}\bigg({^{\T}}{\mathds{I}}_{a+}^{\alpha;\psi}\phi(t)\bigg)\varphi(t)\Delta t
&=& \int_{a}^{b}\bigg(\frac{1}{\Gamma(\alpha)}\int_{a}^{t}\psi^{\Delta}(s)(\psi(t)-\psi(s))^{\alpha-1}
\phi(s) \Delta s\bigg)\varphi(t)\Delta t\\
&=&\int_{a}^{b}\frac{1}{\Gamma(\alpha)}\int_{t}^{b}
\phi(t)\psi^{\Delta}(t)(\psi(s)-\psi(t))^{\alpha-1}\varphi(t)\Delta t\\
&=&\int_{a}^{b}\phi(t)\psi^{\Delta}(t)\,{^{\T}}{\mathds{I}}_{b-}^{\alpha;\psi}
\l(\frac{\varphi(t)}{\psi^{\Delta}(t)}\r) \Delta t.
\end{eqnarray*}
\end{proof}

% ----------------------------------------------

\section{Existence and uniqueness}
\label{sec5}

Now we investigate the question of existence 
and uniqueness of solution to problem \eqref{I*}.

\begin{lemma} 
Let $0<\alpha<1$, $J\subseteq\T$, and $f:J\times\R\to\R$.  
We say that $y(t)$ is a solution of problem \eqref{I*} 
if, and only if, this function is a solution of
\begin{eqnarray}
\label{II}
y(t)=\frac{1}{g^{\T}(\gamma-1,1-\gamma)} \frac{1}{\Gamma(\alpha)}\int_{0}^{t}
\psi^{\Delta}(s)(\psi(t)-\psi(s))^{\alpha-1}f(s,y(s))\Delta s,
\end{eqnarray}
where $g^{\T}(\gamma-1,1-\gamma):=\dfrac{B^{\T}_{0,1}(\gamma-1,
1-\gamma)}{B(\gamma-1,1-\gamma)} $ with $\gamma=\alpha+\beta(1-\alpha)$.
\end{lemma}

\begin{proof} 
Applying the operator ${^{\T}}{\mathds{I}}_{0}^{\alpha;\psi}(\cdot)$  
to both sides of problem \eqref{I*}, using the initial condition 
and Theorem~\ref{th-5}, we have
\begin{eqnarray}
\label{2.5}
y(t)= \frac{1}{g^{\T}(\gamma-1,1-\gamma)} \frac{1}{\Gamma(\alpha)}
\int_{0}^{t}\psi^{\Delta}(s)(\psi(t)-\psi(s))^{\alpha-1} f(s,y(s))\Delta s.
\end{eqnarray}

On the other hand, applying the operator $\dTt(\cdot)$ on both sides 
of \eqref{2.5}, and using Proposition~\ref{prop-I}, we obtain
\begin{eqnarray*}
\dTt y(t)
&=&\dTt\left(\frac{1}{g^{\T}(\gamma-1,1-\gamma)\Gamma(\alpha)} \int_{0}^{t}
\psi^{\Delta}(s)(\psi(t)-\psi(s))^{\alpha-1}f(s,y(s))\,\Delta s\right)\notag\\
&=& f(t,y(t)).
\end{eqnarray*}
Now, taking ${^{\T}}{\mathds{I}}_{0}^{1-\gamma;\psi}(\cdot)$ on both sides 
of Eq.(\ref{2.5}) and using Lemma~\ref{lemma-4}, we get 
${^{\T}}{\mathds{I}}_{0}^{1-\gamma;\psi} y(0)=0$. The proof is complete.
\end{proof}

\begin{proof} (of Theorem~\ref{teorema32})
Let $S$ be the set of right-dense continuous functions 
and $J\subseteq\T$. For $y\in S$, define
\begin{equation*}
\|y\|_{C_{1-\gamma,\psi}}=\sup_{t\in J}\|(\psi(t)-\psi(0))^{1-\gamma}y(t)\|_{C}.
\end{equation*}

Note that $S$ is a Banach space. Define the subset 
$S_{\psi}(\rho)$ and the operator $\Theta$ by
\begin{equation*}
S_{\psi}(\rho)=\big\{x\in S:\|x_s\|_{C_{1-\gamma,\psi}}\leq\rho\big\}
\end{equation*}
and
\begin{equation*}
\Theta(y)= \frac{1}{g^{\T}(\gamma-1,1-\gamma)} \frac{1}{\Gamma(\alpha)} 
\int_{0}^{t}\psi^{\Delta}(s)(\psi(t)-\psi(s))^{\alpha-1}f(s,y(s))\Delta s.    
\end{equation*}

Using Proposition~\ref{Proposition24}, one has
\begin{eqnarray*}
|(\psi(t)-\psi(0))^{1-\gamma}\Theta(y(t))|
&=&\bigg|\frac{(\psi(t)-\psi(0))^{1-\gamma}}{g^{\T}(\gamma-1,1-\gamma)\Gamma(\alpha)}\int_{0}^{t}
\psi^{\Delta}(s)(\psi(t)-\psi(s))^{\alpha-1}f(s,y(s))\Delta s\bigg|\\
&\leq & \frac{(\psi(t)-\psi(0))^{1-\gamma}}{g^{\T}(\gamma-1,1-\gamma)\Gamma(\alpha)}\int_{0}^{t}
\psi^{\Delta}(s)(\psi(t)-\psi(s))^{\alpha-1}|f(s,y(s))|\Delta s\\
&\leq & M\,
\frac{(\psi(t)-\psi(0))^{1-\gamma}}{g^{\T}(\gamma-1,1-\gamma)\Gamma(\alpha)}\int_{0}^{t}
\psi^{\Delta}(s)(\psi(t)-\psi(s))^{\alpha-1}\Delta s.
\end{eqnarray*}
Since $\psi^{\Delta}(s)(\psi(t)-\psi(s))^{\alpha-1}$ 
is an increasing function, by Proposition~\ref{Prop-6} it follows that
\begin{equation*}
\int_{0}^{t}\psi^{\Delta}(s)(\psi(t)-\psi(s))^{\alpha-1}\Delta s
\leq\int_{0}^{t} \psi'(s)(\psi(t)-\psi(s))^{\alpha-1}{\rm d}s.
\end{equation*}

Consequently, one has
\begin{eqnarray*}
|(\psi(t)-\psi(0))^{1-\gamma}\Theta(y(t))|
&\leq & M\,
\frac{(\psi(t)-\psi(0))^{1-\gamma}}{g^{\T}(\gamma-1,1-\gamma)\Gamma(\alpha)}\int_{0}^{t}
\psi'(s)(\psi(t)-\psi(s))^{\alpha-1}{\rm d}s\\
&=&\frac{M}{g^{\T}(\gamma-1,1-\gamma)\Gamma(\alpha)}
\frac{(\psi(t)-\psi(0))^{\alpha}}{\alpha}\,(\psi(t)-\psi(0))^{1-\gamma}\\
&\leq &M\,\frac{(\psi(1)-\psi(0))^{1-\beta(1-\alpha)}}{\Gamma(\alpha+1)},
\end{eqnarray*}
that is,
$$
\|\Theta y\|_{C_{1-\gamma,\psi}}\leq M\,\frac{(\psi(1)
-\psi(0))^{1-\beta(1-\alpha)}}{g^{\T}(\gamma-1,1-\gamma)\Gamma(\alpha+1)}.
$$

Now, we consider
\begin{equation*}
\rho=M\,
\frac{(\psi(1)-\psi(0))^{1-\beta(1-\alpha)}}{g^{\T}(\gamma-1,1-\gamma)\Gamma(\alpha+1)}.
\end{equation*}
We have that $\Theta$ is an operator from $S_{\psi}(\rho)$ to $S_{\psi}(\rho)$. 
Moreover, for $x,y\in S_{\psi}(\rho)$, yields
\begin{equation}
\label{star*}
\begin{split}
\|(\psi(t)-\psi(0))^{1-\gamma}(\Theta x(t)- \Theta y(t))\| 
&\leq \frac{(\psi(t)-\psi(0))^{1-\gamma}}{g^{\T}(\gamma-1,1-\gamma)\Gamma(\alpha)}
\int_{0}^{t}\psi^{\Delta}(s)(\psi(t)-\psi(s))^{\alpha-1}|f(s,x(s))-f(s,y(s))|\Delta s\\
&\leq {L}\,\frac{\|x-y\|_{\infty}(\psi(t)-\psi(0))^{1
-\gamma}}{g^{\T}(\gamma-1,1-\gamma)\Gamma(\alpha)}
\int_{0}^{t}\psi^{\Delta}(s)(\psi(t)-\psi(s))^{\alpha-1}\Delta s\\
&\leq {L} \frac{\|x-y\|_{C_{1-\gamma,\psi}}(\psi(t)
-\psi(0))^{1-\gamma}}{g^{\T}(\gamma-1,1-\gamma)\Gamma(\alpha)}
\int_{0}^{t}\psi'(s)(\psi(t)-\psi(s))^{\alpha-1}{\rm d}s\\
&= {L}\,\frac{\|x-y\|_{C_{1-\gamma,\psi}}}{g^{\T}(\gamma-1,1-\gamma)
\Gamma(\alpha)}\frac{(\psi(t)-\psi(0))^{1-\beta(1-\alpha)}}{\alpha}\\
&\leq {L}\,\frac{(\psi(1)-\psi(0))^{\alpha}}{g^{\T}(\gamma-1,
1-\gamma)\Gamma(\alpha+1)}\|x-y\|_{C_{1-\gamma,\psi}}.
\end{split}
\end{equation}

It follows that
\begin{equation}
\label{A**}
\|\Theta x-\Theta y\|_{C_{1-\gamma,\psi}}
\leq {L}\,\frac{(\psi(1)-\psi(0))^{\alpha}}{g^{\T}(\gamma-1,1-\gamma)\Gamma(\alpha+1)}
\|x-y\|_{C_{1-\gamma,\psi}}. 
\end{equation}

Indeed, evaluating the supremum, for $t\in[0,1]$ of both sides of \eqref{star*}, 
and using the definition of the norm in the weighted space, 
we have \eqref{A**}. If $\displaystyle {L}\,\frac{(\psi(1)
-\psi(0))^{\alpha}}{g^{\T}(\gamma-1,1-\gamma)\Gamma(\alpha+1)}<1$, 
then this will be a contraction map and we obtain the desired existence 
and uniqueness of solution to problem \eqref{I*}.
\end{proof}

\begin{remark}
\label{remark-2}
(i) Taking the limit $\beta\to 0$ in \eqref{I*}, 
we obtain a $\psi$-Caputo fractional derivative problem on times scales. 
Using Theorem~\ref{teorema32}, a solution to such problem exists and is unique.
(ii) Taking $\beta\to 1$ in \eqref{I*}, we get a corresponding problem on time scales in the 
$\psi$-Riemann--Liouville fractional derivative sense. Under the conditions of Theorem~\ref{teorema32}, 
we conclude that such problem admits a unique solution. (iii) From the choice of $g(\cdot)$, 
we obtain numerous particular cases for problem \eqref{I*}, for which our Theorem~\ref{teorema32}
provides a sufficient condition for existence of a unique solution.
\end{remark}

\begin{proof} (of Theorem~\ref{th-3.4})
We prove the result in four steps. 

\noindent {\rm Step 1:} $\Theta$ is continuous.  

Consider the a sequence $y_n$ that $y_n\to y$ in $C(J,\R)$. So, for each $t\in J$, yields
\begin{eqnarray*}
&&|(\psi(t)-\psi(0))^{1-\gamma}(\Theta(y_n)(t)
-\Theta(y)(t))|\\
&\leq &\frac{(\psi(t)-\psi(0))^{1-\gamma}}{g^{\T}(\gamma-1,1-\gamma)\Gamma(\alpha)}\int_{0}^{t}
\psi^{\Delta}(s)(\psi(t)-\psi(s))^{\alpha-1}|f(s,y_n(s))-f(s,y(s))|\Delta s\\
&\leq &\frac{(\psi(t)-\psi(0))^{1-\gamma}}{g^{\T}(\gamma-1,1-\gamma)\Gamma(\alpha)}\int_{0}^{t}
\psi^{\Delta}(s)(\psi(t)-\psi(s))^{\alpha-1}\sup_{s\in J}|f(s,y_n(s))-f(s,y(s))|\Delta s\\
&\leq &\frac{\|f(\cdot,y_n(\cdot))-f(\cdot,y(\cdot))\|_{C_{1-\gamma,\psi}}}{g^{\T}(\gamma-1,1-\gamma)\Gamma(\alpha)}
\int_{0}^{t}\psi^{\Delta}(s)(\psi(t)-\psi(s))^{\alpha-1}\Delta s\\
&\leq & \frac{\|f(\cdot,y_n(\cdot))-f(\cdot,y(\cdot))\|_{C_{1-\gamma,\psi}}}{g^{\T}(\gamma-1,1-\gamma)\Gamma(\alpha)}
\int_{0}^{t}\psi'(s)(\psi(t)-\psi(s))^{\alpha-1}{\rm d}s\\
&\leq& \frac{\|f(\cdot,y_n(\cdot))-f(\cdot,y(\cdot))\|_{C_{1-\gamma,\psi}}}{g^{\T}(\gamma-1,1-\gamma)\Gamma(\alpha)}\,
(\psi(1)-\psi(0))^{\alpha}.
\end{eqnarray*}

Since $f$ is a continuous function, we obtain that
\begin{equation*}
\|\Theta y_n-\Theta y\|_{C_{1-\gamma,\psi}}
\leq\frac{(\psi(1)-\psi(0))^{\alpha}}{g^{\T}(\gamma-1,1-\gamma)\Gamma(\alpha+1)} 
\|f(\cdot,y_n(\cdot))-f(\cdot,y(\cdot))\|_{C_{1-\gamma,\psi}}\to 0
\end{equation*}
as $n\to\infty$.

\noindent{\rm Step 2:} The map $\Theta$ sends bounded sets 
into bounded sets in $C(J,\R)$. To see that, it is enough 
to show that for any $\rho$, there exists a positive constant $\ell$ 
such that for each $y\in B_{\rho}=\{y\in C_{1-\gamma,\psi}(J,\R)
:\|y\|_{C_{1-\gamma,\psi}}\leq\rho\}$ we have $\|\Theta y\|_{C_{1-\gamma,\psi}}\leq\ell$. 

Indeed, by hypothesis, for each $t\in J$ one has
\begin{eqnarray*}
\|(\psi(t)-\psi(0))^{1-\gamma}\Theta y(t)\|&\leq &\frac{(\psi(t)
-\psi(0))^{1-\gamma}}{g^{\T}(\gamma-1,1-\gamma)\Gamma(\alpha)}
\int_{0}^{t}\psi^{\Delta}(s)(\psi(t)-\psi(s))^{\alpha-1}|f(s,y(s))|\Delta s\\
&\leq & M\,\frac{(\psi(t)-\psi(0))^{1-\gamma}}{g^{\T}(\gamma-1,1-\gamma)\Gamma(\alpha)}\int_{0}^{t}
\psi^{\Delta}(s)(\psi(t)-\psi(s))^{\alpha-1}\Delta s\\
&\leq & M\,\frac{(\psi(t)-\psi(0))^{1-\gamma}}{g^{\T}(\gamma-1,1-\gamma)\Gamma(\alpha)}
\int_{0}^{t} \psi'(s)(\psi(t)-\psi(s))^{\alpha-1}{\rm d}s\\
&\leq &\frac{M}{g^{\T}(\gamma-1,1-\gamma)\Gamma(\alpha+1)}(\psi(1)-\psi(0))^{1-\beta(1-\alpha)}=\ell.
\end{eqnarray*}

\noindent{\rm Step 3:} The map $\Theta$ sends bounded sets 
into equicontinuous sets of $C(J,\R)$. Let $t_1,t_2\in J$, $t_1<t_2$, 
$B_{\rho}$ be a bounded set of $C(J,\R)$ as in \rm{Step 2}, 
and $y\in B_{\rho}$. Then,
\begin{eqnarray*}
|(\Theta y)(t_2)-(\Theta y)(t_1)|&\leq 
&\frac{(\psi(t_2)-\psi(0))^{1-\gamma}}{g^{\T}(\gamma-1,1-\gamma)\Gamma(\alpha)}
\int_{0}^{t_1}\psi^{\Delta}(s)(\psi(t_1)-\psi(s))^{\alpha-1}|f(s,y(s))|\Delta s\\
&\quad&-\frac{(\psi(t_1)-\psi(0))^{1-\gamma}}{g^{\T}(\gamma-1,1-\gamma)\Gamma(\alpha)}
\int_{0}^{t_2}\psi^{\Delta}(s)(\psi(t_2)-\psi(s))^{\alpha-1}|f(s,y(s))|\Delta s\\
&\leq &M\,\frac{(\psi(t_2)-\psi(0))^{1-\gamma}}{g^{\T}(\gamma-1,1-\gamma)\Gamma(\alpha)}
\int_{0}^{t_1}\psi^{\Delta}(s)[(\psi(t_1)-\psi(s))^{\alpha-1}-(\psi(t_2)-\psi(s))^{\alpha-1}]\Delta s\\
&\quad & + M\,\frac{(\psi(t_1)-\psi(0))^{1-\gamma}}{g^{\T}(\gamma-1,1-\gamma)\Gamma(\alpha)}
\int_{t_1}^{t_2}\psi^{\Delta}(s)(\psi(t_2)-\psi(s))^{\alpha-1}\Delta s\\
&\leq & M\,\frac{(\psi(t_2)-\psi(0))^{1-\gamma}}{g^{\T}(\gamma-1,1-\gamma)\Gamma(\alpha)}
\int_{0}^{t_1}\psi'(s)[(\psi(t_1)-\psi(s))^{\alpha-1}-(\psi(t_2)-\psi(s))^{\alpha-1}]{\rm d}s\\
&\quad& +M\,\frac{(\psi(t_1)-\psi(0))^{1-\gamma}}{g^{\T}(\gamma-1,1-\gamma)\Gamma(\alpha)}\int_{t_1}^{t_2}
\psi'(s)(\psi(t_2)-\psi(s))^{\alpha-1}{\rm d}s\\
&\leq & M\,\frac{(\psi(t_2)-\psi(0))^{1-\gamma}}{g^{\T}(\gamma-1,1-\gamma)\Gamma(\alpha)}
\big[(\psi(t_2)-\psi(t_1))^{\alpha}+(\psi(t_1)-\psi(0))^{\alpha}-(\psi(t_2)-\psi(0))^{\alpha}\big]\\
&\quad& + M\,\frac{(\psi(t_1)-\psi(0))^{1-\gamma}}{g^{\T}(\gamma-1,1-\gamma)\Gamma(\alpha)}\,(\psi(t_2)-\psi(t_1))^{\alpha}\\
&=&\frac{2 M}{g^{\T}(\gamma-1,1-\gamma)\Gamma(\alpha+1)}\,(\psi(t_2)-\psi(0))^{\alpha}\big[(\psi(t_2)-\psi(0))^{1-\gamma}
-(\psi(t_1)-\psi(0))^{1-\gamma}\big]\\
&\quad&+\frac{M}{g^{\T}(\gamma-1,1-\gamma)\Gamma(\alpha+1)}\,(\psi(t_2)-\psi(0))^{1-\gamma}\big[\,(\psi(t_1)-\psi(0))^{\alpha}
-\,(\psi(t_2)-\psi(0))^{\alpha}\big]\rightarrow 0
\end{eqnarray*}
as $t_1\to t_2$. In this sense, we have $\Theta:C_{1-\gamma;\psi}(J,\R)\to C_{1-\gamma,\psi}(J,\R)$ 
is continuous and completely continuous (consequence of steps 1 to 3 and Arzela-Ascoli theorem).

\noindent{\rm Step 4:} A priory bounds. Now it remains to show that
$\Omega=\{y\in C(J,\R):y=\lambda \Theta(y);0<\lambda<1\}$ 
is a bounded set. Let $y\in\Omega$. Then, $y=\lambda \Theta(y)$ 
for some $0<\lambda<1$. Thus, for each $t\in J$, it yields that
\begin{equation*}
y(t)=\lambda\Bigg[\frac{1}{g^{\T}(\gamma-1,1-\gamma)\Gamma(\alpha)}
\int_{0}^{t}\psi^{\Delta}(s) (\psi(t)-\psi(s))^{\alpha-1}f(s,y(s))\Delta s\Bigg].    
\end{equation*}

Thus, we conclude via the estimate in Step~2. 
In this sense, $\Theta$ has a fixed point, hence a solution to the problem \eqref{I*}
(Schauder's fixed point theorem). 
\end{proof}

% ---------------------------------------- 

\section{Controllability}
\label{sec6}

In this section, we investigate the question of controllability 
for \eqref{eqI}. For this, first, we present the concept of 
controllability and the integral equation that is equivalent 
to the problem to be discussed. 

\begin{definition} 
We say that \eqref{eqI} is controllable on $J$ if, for any given initial state $y_0$
and any given final state $\bar{y}$, there exists a piecewise right-dense continuous function 
$u\in L^{2}(J,U)$ such that the solution $y$ of \eqref{eqI} satisfies $y(1)=\bar{y}$.
\end{definition}

\begin{theorem}  
A function $f\in C(J,\mathbb{R})$ is a solution of \eqref{eqI} if, and only if, 
this function is a solution of the following integral equation:
\begin{eqnarray*}
y(t)= \frac{1}{g^{\T}(\gamma-1,1-\gamma)} \frac{1}{\Gamma(\alpha)}
\int_{0}^{t}\psi^{\Delta}(s)(\psi(t)-\psi(s))^{\alpha-1}\big(f(s,y(s))+(Bu)(s)\big)\Delta s,
\end{eqnarray*}
$t\in[0,1]=J\subset\T$.
\end{theorem} 

\begin{proof} 
Applying the operator ${^{\T}}{\mathds{I}}_{0}^{\alpha;\psi}(\cdot)$  
to both sides of problem \eqref{eqI}, using the initial condition 
and Theorem~\ref{th-5}, we have
\begin{eqnarray}
\label{2.54}
y(t)= \frac{1}{g^{\T}(\gamma-1,1-\gamma)} \frac{1}{\Gamma(\alpha)}
\int_{0}^{t}\psi^{\Delta}(s)(\psi(t)-\psi(s))^{\alpha-1}\big(f(s,y(s))+(Bu)(s)\big)\Delta s.
\end{eqnarray}

On the other hand, applying the operator $\dTt(\cdot)$ on both sides 
of \eqref{2.54}, and using Proposition~\ref{prop-I}, we obtain
\begin{eqnarray*}
\dTt y(t)
&=&\dTt\left(\frac{1}{g^{\T}(\gamma-1,1-\gamma)\Gamma(\alpha)} \int_{0}^{t}
\psi^{\Delta}(s)(\psi(t)-\psi(s))^{\alpha-1}(f(s,y(s))+(Bu)(s))\Delta s\right)\notag\\
&=& f(t,y(t))+(Bu)(t).
\end{eqnarray*}
Now, taking ${^{\T}}{\mathds{I}}_{0}^{1-\gamma;\psi}(\cdot)$ on both sides of 
the Eq.(\ref{2.54}) and using Lemma~\ref{lemma-4}, we get 
${^{\T}}{\mathds{I}}_{0}^{1-\gamma;\psi} y(0)=0$.
The proof is complete.
\end{proof}

\begin{lemma}
\label{lemma 3.4}
Let the assumptions ${\rm (A_1)}$--${\rm (A_4)}$ be satisfied and 
$y(0)\in\R$ be an arbitrary point. Then the solution $y(t)$ 
of system \eqref{eqI} on $[0,1]$ is defined by the control function
$$
u(t)=(\mathscr{W}_{\alpha})^{-1}\bigg[y_1- \frac{1}{g^{\T}(\gamma-1,1-\gamma)\Gamma(\alpha)}
\int_{0}^{t}\psi^{\Delta}(s)(\psi(t)-\psi(s))^{\alpha-1}f(s,y(s))\Delta s\bigg], 
$$
$t\in[0,1]$. Moreover, the control function $u(t)$ has an estimate 
$\|u(t)\|\leq M_{u}^{\circ}$ with
\begin{equation*}
M_{u}^{\circ}=|y_1|+\frac{M}{g^{\T}(\gamma-1,1-\gamma)\Gamma(\alpha+1)}(\psi(t)-\psi(0))^{\alpha}.
\end{equation*}
\end{lemma}

\begin{proof} 
Let $y(t)$ be a solution of system \eqref{eqI} 
on $[0,1]\subset\T$ defined by \eqref{2.5}. Then, 
\begin{eqnarray*}
y(t)&=&\frac{1}{g^{\T}(\gamma-1,1-\gamma)\Gamma(\alpha)}\int_{0}^{t}\psi^{\Delta}(s)(\psi(t)-\psi(s))^{\alpha-1}
\big[f(s,y(s))+(Bu)(s)\big]\Delta s\\
&=&\frac{1}{g^{\T}(\gamma-1,1-\gamma)\Gamma(\alpha)}\int_{0}^{t}\psi^{\Delta}(s)(\psi(t)-\psi(s))^{\alpha-1}f(s,y(s))\Delta s\\
&\quad&+\frac{1}{g^{\T}(\gamma-1,1-\gamma)\Gamma(\alpha)}\int_{0}^{t}\psi^{\Delta}(s)(\psi(t)-\psi(s))^{\alpha-1}
B(\mathscr{W}_{\alpha})^{-1} \\&& \bigg[y_1
-\frac{1}{g^{\T}(\gamma-1,1-\gamma)\Gamma(\alpha)}\int_{0}^{t}\psi^{\Delta}(\xi)(\psi(t)-\psi(\xi))^{\alpha-1}
f(\xi,y(\xi))\Delta\xi\bigg]\Delta s\\
&=&\frac{1}{g^{\T}(\gamma-1,1-\gamma)\Gamma(\alpha)}\int_{0}^{t}\psi^{\Delta}(s)(\psi(t)-\psi(s))^{\alpha-1}f(s,y(s))\Delta s\\
&\quad&+\mathscr{W}_{\alpha}(\mathscr{W}_{\alpha})^{-1}\bigg[y_1
+\frac{1}{g^{\T}(\gamma-1,1-\gamma)\Gamma(\alpha)}\int_{0}^{t}\psi^{\Delta}(\xi)(\psi(t)-\psi(\xi))^{\alpha-1}
f(\xi,y(\xi))\Delta\xi\bigg]\\
&=&y_1.
\end{eqnarray*}

In this sense, we have the estimative
\begin{eqnarray*}
|u(t)|&=&\bigg|(\mathscr{W}_{\alpha})^{-1}\bigg(y_1-\frac{1}{g^{\T}(\gamma-1,
1-\gamma)\Gamma(\alpha)}\int_{0}^{t}\psi^{\Delta}(s)(\psi(t)-\psi(s))^{\alpha-1}f(s,y(s))\Delta s\bigg)\bigg|\\
&\leq &|(\mathscr{W}_{\alpha})^{-1}\bigg(|y_1|
+\frac{1}{g^{\T}(\gamma-1,1-\gamma)\Gamma(\alpha)}\int_{0}^{t}\psi^{\Delta}(s)
(\psi(t)-\psi(s))^{\alpha-1}|f(s,y(s))|\Delta s\bigg)\\
&\leq & M_{W}^{\circ}\bigg(|y_1|+\frac{M}{g^{\T}(\gamma-1,1-\gamma)\Gamma(\alpha)}
\int_{0}^{t}\psi^{\Delta}(s)(\psi(t)-\psi(s))^{\alpha-1}\Delta s\bigg)\\
&\leq & M_{W}^{\circ}\bigg(|y_1|+\frac{M}{g^{\T}(\gamma-1,1-\gamma)\Gamma(\alpha)}
\int_{0}^{t}\psi'(s)(\psi(t)-\psi(s))^{\alpha-1}{\rm d}s\bigg)\\
&=&M_{W}^{\circ}\bigg(|y_1|+\frac{M}{g^{\T}(\gamma-1,1-\gamma)\Gamma(\alpha+1)}
\,(\psi(t)-\psi(0))^{\alpha}\bigg)\\
&=& M_{u}^{\circ}.
\end{eqnarray*}
Therefore, the proof is complete.
\end{proof}

To finalize the main results, next we present the proof 
that guarantees that the system \eqref{eqI} is controllable.

\begin{proof} (of Theorem \ref{th-4.2})
Consider the subset ${D}_{\psi,\delta}\subseteq C_{1-\gamma,\psi}(J,\R)$ as follows:
\begin{equation*}
{D}_{\psi,\delta}=\big\{x\in C_{1-\gamma,\psi}(J,\R):\|x\|_{C_{1-\gamma,\psi}}\leq\delta\big\}.
\end{equation*}

We define the operator ${K}:{D}_{\psi,\delta}\to {D}_{\psi,\delta}$ as
\begin{equation*}
({K}y)(t)= \frac{1}{g^{\T}(\gamma-1,1-\gamma)\Gamma(\alpha)}
\int_{0}^{t}\psi^{\Delta}(s)(\psi(t)-\psi(s))^{\alpha-1}
\big(f(s,y(s))+(Bu)(s)\big)\Delta s.    
\end{equation*}

Note that the operator ${K}$ is well defined and the fixed points of ${K}$ are solutions to \eqref{eqI}. 
Indeed, $x\in {D}_{\psi,\delta}$ is a solution of \eqref{eqI} if, and only if, it is a solution 
of the operator equation $x={K}x$. Therefore, the existence of a solution of \eqref{eqI} 
is equivalent to determine a positive constant $\delta$ such that ${K}$ 
has a fixed point on ${D}_{\psi,\delta}$. We decompose the operator ${K}$ into two operators 
${K}_1$ and ${K}_2$, ${K}={K}_1+{K}_2$ on ${D}_{\psi,\delta}$, where
\begin{eqnarray*}
({K}_1 y)(t)=\frac{1}{g^{\T}(\gamma-1,1-\gamma)\Gamma(\alpha)}\int_{0}^{t}\psi^{\Delta}(s)
(\psi(t)-\psi(s))^{\alpha-1} Bu(s)\Delta s, \quad t\in[0,1]=J\subset\T
\end{eqnarray*} 
and
\begin{eqnarray*}
({K}_2 y)(t)= \frac{1}{g^{\T}(\gamma-1,1-\gamma)\Gamma(\alpha)}
\int_{0}^{t}\psi^{\Delta}(s)(\psi(t)-\psi(s))^{\alpha-1} f(s,u(s))\Delta s.
\end{eqnarray*}

\noindent \rm{Step 1:} The operator ${K}_1$ maps ${D}_{\psi,\delta}$ into itself. 
For each $t\in J$ and $x\in {D}_{\psi,\delta}$, it follows from Lemma~\ref{lemma 3.4} that
\begin{eqnarray*}
&&\big\|(g(t)-\psi(0))^{1-\gamma}(K_1 x)(t)\big\|\notag \\
&=&\Bigg\|\frac{(\psi(t)-\psi(0))^{1-\gamma}}{g^{\T}(\gamma-1,1-\gamma)\Gamma(\alpha)}
\int_{0}^{t}\psi^{\Delta}(s) (\psi(t)-\psi(s))^{\alpha-1}(Bu)(s)\Delta s\Bigg\|\\
&\leq &\frac{(\psi(t)-\psi(0))^{1-\gamma}}{g^{\T}(\gamma-1,1-\gamma)\Gamma(\alpha)}
\int_{0}^{t}\psi^{\Delta}(s) (\psi(t)-\psi(s))^{\alpha-1}\|B\| \|u(s)\|\Delta s\\
&\leq & \frac{(\psi(t)-\psi(0))^{1-\gamma}}{g^{\T}(\gamma-1,1-\gamma)\Gamma(\alpha)}\, M_{B}
\int_{0}^{t}\psi^{\Delta}(s)(\psi(t)-\psi(s))^{\alpha-1}M_{W}^{\circ}\\
&\quad & \times \Bigg(|y_1|+\frac{M}{g^{\T}(\gamma-1,1-\gamma)\Gamma(\alpha+1)}(\psi(s)-\psi(0))^{\alpha}\Bigg)\Delta s\\
&\leq& \frac{(\psi(1)-\psi(0))^{1-\gamma}}{g^{\T}(\gamma-1,1-\gamma)\Gamma(\alpha)}\,M_{W}^{\circ}\,M_{B}
\Bigg(|y_1| +\frac{M}{g^{\T}(\gamma-1,1-\gamma)\Gamma(\alpha+1)}(\psi(1)-\psi(0))^{\alpha}\Bigg)\\
&\quad& \int_{0}^{t}
\psi'(s)(\psi(t)-\psi(s))^{\alpha-1}{\rm d}s\\
&\leq& \frac{(\psi(1)-\psi(0))^{1-\beta(1-\alpha)}}{g^{\T}(\gamma-1,1-\gamma)\Gamma(\alpha+1)}\,M_{W}^{\circ}
\, M_{B}\Bigg(|y_1|+\frac{M}{g^{\T}(\gamma-1,1-\gamma)\Gamma(\alpha+1)}(\psi(1)-\psi(0))^{\alpha}\Bigg)\\
&\leq & \delta,
\end{eqnarray*}
which implies that $\|{K}_1 y\|_{C_{1-\gamma,\psi}}\leq\delta$. 
Thus, ${K}_1$ maps ${D}_{\psi,\delta}$ into itself.

\noindent \rm{Step 2:} The operator ${K}_2$ is continuous. Let $\{y_n\}$ be a sequence 
in ${D}_{\psi,\delta}$ satisfying $y_n\to y$ as $n\to\infty.$ Then, for each $t\in J$, one has
\begin{eqnarray*}
&&\Big\|(\psi(t)-\psi(0))^{1-\gamma}\big(({K}_2 y_n)(t)-({K}_2 y)(t)\big)\Big\|_{C}\\
&\leq& \frac{(\psi(t)-\psi(0))^{1-\gamma}}{g^{\T}(\gamma-1,1-\gamma)\Gamma(\alpha)}\int_{0}^{t}
\psi^{\Delta}(s)(\psi(t)-\psi(s))^{\alpha-1}\|f(s,y_n(s))-f(s,y(s))\|\Delta s\\
&\leq &\frac{(\psi(t)-\psi(0))^{1-\gamma}}{g^{\T}(\gamma-1,1-\gamma)\Gamma(\alpha)}\int_{0}^{t}g'(s)
(\psi(t)-\psi(s))^{\alpha-1}\|f(s,y_n(s))-f(s,y(s))\|{\rm d}s\\
&\leq &\|f(\cdot\,,y_n(\cdot))-f(\cdot\,,y(\cdot))\|_{C_{1-\gamma,\psi}}
\frac{(\psi(t)-\psi(0))^{1-\beta(1-\alpha)}}{g^{\T}(\gamma-1,1-\gamma)\Gamma(\alpha+1)}.
\end{eqnarray*}

By the Lebesgue dominated convergence theorem, we know that
$\|{K}_2 y_n-{K}_2 y\|_{C_{1-\gamma,\psi}}\to 0$ 
as $n\to\infty$. This means that ${K}_2$ is continuous.

\noindent \rm{Step 3:} Now we show that ${K}_2({D}_{\psi,\delta})\subset {D}_{\psi,\delta}$.
We prove this by contradiction, supposing that there exists a function 
$\eta(\cdot)\in {D}_{\psi,\delta}$ such that $\|({K}_2 y)\|_{C_{1-\gamma,\psi}}>\delta$. 
Thus, under such assumption, for each $t\in J$ we get
\begin{eqnarray*}
\delta &<&\big\|(\psi(t)-\psi(0))^{1-\gamma}({K}_2\eta)(t)\big\|\\
&\leq &\frac{(\psi(t)-\psi(0))^{1-\gamma}}{g^{\T}(\gamma-1,1-\gamma)\Gamma(\alpha)}
\int_{0}^{t}\psi^{\Delta}(s)(\psi(t)-\psi(s))^{\alpha-1}\|f(s,\eta(s))\|\Delta s\\
&\leq &\frac{(\psi(t)-\psi(0))^{1-\gamma}}{g^{\T}(\gamma-1,1-\gamma)\Gamma(\alpha)}
\int_{0}^{t}\psi'(s)(\psi(t)-\psi(s))^{\alpha-1}\|f(s,\eta(s))\|{\rm d}s\\ 
&\leq &M\,
\frac{(\psi(1)-\psi(0))^{1-\beta(1-\alpha)}}{g^{\T}(\gamma-1,1-\gamma)\Gamma(\alpha+1)}\,\delta.
\end{eqnarray*}
Dividing both sides by $\delta$, and taking the limit as ${K}\to\infty$, we get
\begin{equation*}
M\,\frac{(\psi(1)-\psi(0))^{1-\beta(1-\alpha)}}{g^{\T}(\gamma-1,1-\gamma)\Gamma(\alpha+1)}\geq{1},
\end{equation*}
which contradicts \eqref{4.1}. This shows that 
${K}_2({D}_{\psi,\delta})\subset {D}_{\psi,\delta}$.

\noindent \rm{Step 4:} Now we show that ${K}_2({D}_{\psi,\delta})$ is bounded and equicontinous.
From Step 3, it is clear that ${K}_2({D}_{\psi,\delta})$ is bounded. It remains to show that 
${K}_2({D}_{\psi,\delta})$ is equicontinuous. Indeed, we have
\begin{eqnarray*}
&&\big\|(\psi(t_2)-\psi(0))^{1-\gamma}({K}_2 y)(t_2)-(\psi(t_1)-\psi(0))^{1-\gamma}({K}_1 y)(t_1)\big\|\\
&=&\Bigg\|\frac{(\psi(t_2)-\psi(0))^{1-\gamma}}{g^{\T}(\gamma-1,1-\gamma)\Gamma(\alpha)}\int_{0}^{t_2}
\psi^{\Delta}(s)(\psi(t_2)-\psi(s))^{\alpha-1}f(s,y(s))\Delta s\\
&\qquad &-\frac{(\psi(t_1)-\psi(0))^{1-\gamma}}{g^{\T}(\gamma-1,1-\gamma)\Gamma(\alpha)}\int_{0}^{t_1}
\psi^{\Delta}(s)(\psi(t_1)-\psi(s))^{\alpha-1}f(s,y(s))\Delta s\Bigg\|\\
&\leq&\Bigg\|\frac{(\psi(t_2)-\psi(0))^{1-\gamma}
-(\psi(t_1)-\psi(0))^{1-\gamma}}{g^{\T}(\gamma-1,1-\gamma)\Gamma(\alpha)} 
\int_{0}^{t_2}\psi^{\Delta}(s)(\psi(t_2)-\psi(s))^{\alpha-1}f(s,y(s))\Delta s\Bigg\|\\
&\quad& +\Bigg\|\frac{1}{g^{\T}(\gamma-1,1-\gamma)
\Gamma(\alpha)}(\psi(t_1)-\psi(0))^{1-\gamma}\int_{t_1}^{t_2}\psi^{\Delta}(s)
(\psi(t_2)-\psi(s))^{\alpha-1}f(s,y(s))\Delta s\Bigg\|\\
&\quad& +\Bigg\|\frac{1}{g^{\T}(\gamma-1,1-\gamma)\Gamma(\alpha)}
(\psi(t_1)-\psi(0))^{1-\gamma}\int_{0}^{t_1}\psi^{\Delta}(s)
(\psi(t_1)-\psi(s))^{\alpha-1}f(s,y(s))\Delta s\Bigg\|\\
&\leq& \frac{(\psi(t_2)-\psi(0))^{1-\gamma}-(\psi(t_1)
-\psi(0))^{1-\gamma}}{g^{\T}(\gamma-1,1-\gamma)\Gamma(\alpha)}\,M
\int_{0}^{t_2}g'(s)(\psi(t_2)-\psi(s))^{\alpha-1}{\rm d}s\\
&\quad& +\frac{(\psi(t_1)-\psi(0))^{1-\gamma}}{g^{\T}(\gamma-1,1-\gamma)\Gamma(\alpha)}\,M
\int_{t_1}^{t_2}\psi'(s)(\psi(t_2)-\psi(s))^{\alpha-1}{\rm d}s\\
&\quad& +\frac{(\psi(t_1)-\psi(0))^{1-\gamma}}{g^{\T}(\gamma-1,1-\gamma)\Gamma(\alpha)}\,M
\int_{0}^{t_1}\psi'(s)\big[(\psi(t_2)-\psi(s))^{\alpha-1}-(\psi(t_1)-\psi(s))^{\alpha-1}\big]{\rm d}s\\
\end{eqnarray*}

\begin{eqnarray*}
&\leq& \frac{(\psi(t_2)-\psi(0))^{1-\gamma}-(\psi(t_1)
-\psi(0))^{1-\gamma}}{g^{\T}(\gamma-1,1-\gamma)\Gamma(\alpha+1)}\,M
(\psi(t_2)-\psi(0))^{\alpha}\\
&\quad& +\frac{(\psi(t_1)-\psi(0))^{1-\gamma}}{g^{\T}(\gamma-1,
1-\gamma)\Gamma(\alpha+1)}\,M(\psi(t_2)-\psi(t_1))^{\alpha}\\
&\quad& + \frac{(\psi(t_1)-\psi(0))^{1-\gamma}}{g^{\T}(\gamma-1,
1-\gamma)\Gamma(\alpha+1)}\,M\big[(\psi(t_2)-\psi(0))^{\alpha}
-(\psi(t_1)-\psi(0))^{\alpha}\big]\to 0
\end{eqnarray*}
as $t_2\to t_1$. It follows that
$\|{K}_2 y-{K}_1 y\|_{C_{1-\gamma,\psi}}\to 0$ as $t_2\to t_1$.
Hence, ${K}_2({D}_{\psi,\delta})$ is equicontinuous.

As a consequence of Steps 2--4, together with the Arzel\`{a}-Ascoli theorem, 
one has that ${K}_2$ is compact. Hence, from Steps 1-4 and Lemma~\ref{lemma3.2}, 
we conclude that ${K}={K}_1+{K}_2$ is continuous and takes bounded sets into bounded sets. 
Also, one can verify the validity of $\mu\big({K}_2({D}_{\psi,\delta})\big)=0$ since 
${K}_2({D}_{\psi,\delta})$ is relatively compact. It follows from the inclusion 
${K}_1({D}_{\psi,\delta})\subset {D}_{\psi,\delta}$ and the equality 
$\mu\big({K}_2({D}_{\psi,\delta})\big)=0$ that
\begin{equation*}
\mu\big({K}({D}_{\psi,\delta})\big)\leq\mu\big({K}_1({D}_{\psi,\delta})\big)
+ \mu\big({K}_2({D}_{\psi,\delta})\big)\leq\mu({D}_{\psi,\delta})
\end{equation*}
for every bounded set ${D}_{\psi,\delta}$ 
of $C_{1-\gamma,\psi}(J,\R)$ with $\mu({D}_{\psi,\delta})>0$.
Since ${K}({D}_{\psi,\delta})\subset {D}_{\psi,\delta}$ for the convex, closed and bounded set 
${D}_{\psi,\delta}$ of $C_{1-\gamma,\psi}(J,\R)$, all conditions of the Sadovskii fixed point theorem
are satisfied and we conclude that the operator ${K}$ has a fixed point $x\in {D}_{\psi,\delta}$ 
that is a solution of \eqref{eqI} with $y(1)=y_1$. Therefore, \eqref{eqI} is controllable on $J$.
\end{proof}

% ----------------------------------------

\section{Conclusion and future work}

We presented a new version for the $\psi$-Hilfer fractional derivative, 
in the sense of time scales, proving the fundamental properties of this new derivative. 
The respective proofs were presented in detail and discussed. On the other hand, 
the theory of differential equations on time scales is of great relevance in several areas. 
So, in this sense, in order to present an approach on dynamic equations on time scales
via $\psi$-Hilfer fractional derivatives, we investigated the existence, 
uniqueness, and controllability of solution 
for the problems (\ref{I*}) and (\ref{eqI}), respectively. 

Although we were able to obtain some important results, questions arose 
during the work for future work. These include:
\begin{enumerate}
\item The discussion of type I and type II generalized Leibniz rules 
for the $\psi$-Hilfer fractional derivative in time scales.

\item If one presents a time-scale version of the $\psi$-Hilfer 
fractional derivative of variable order, will the semigroup property also not hold? 
What is lost and what is gained?

\item Is it possible to discuss issues of continuous 
dependence on data and issues of reachability
for fractional dynamic equations on time scales?
    
\item Are $\psi$-Hilfer fractional derivatives in time scales 
relevant to discuss some mathematical modeling problems 
via Laplace transforms in time-scales?
\end{enumerate}

Numerous other issues, involving fractional derivatives 
and differential equations on time scales, can be discussed from our work
and we trust they will emerge naturally over time.

% ----------------------------------------

\section*{Acknowledgements}

This work was supported by \emph{Funda\c{c}\~{a}o 
Cearense de Apoio ao Desenvolvimento Cient\'{\i}fico e Tecnol\'{o}gico} (FUNCAP) 
of Brazil, project BP4-00172-00054.02.00/20; and by
\emph{Funda\c{c}\~{a}o para a Ci\^{e}ncia e a Tecnologia} (FCT) of Portugal,
project UIDB/04106/2020 (CIDMA). The authors would like to thank the referees 
for their careful reading of the work and its many useful suggestions.

% ----------------------------------------

% ----------------------------------------

\end{document}